\documentclass{amsart}
\usepackage{verbatim}
\usepackage{amscd}
\usepackage{amsmath}
\usepackage{amssymb}
\usepackage{latexsym}

\theoremstyle{definition} 

\theoremstyle{plain}      

\newcommand{\rp}{\mathbb{RP}}
\newcommand{\re}{\mathbb{R}}
\newcommand{\co}{\mathbb{C}}
\newcommand{\sph}{\mathbb{S}}
\newcommand{\pgl}[1]{\mathrm{PGL}(#1,\mathbb{R})}
\newcommand{\Sl}[1]{\mathrm{SL}(#1,\mathbb{R})}
\newcommand{\na}{\nabla}
\newcommand{\sfrac}[2]{{\textstyle \frac{#1}{#2}}}

\newcommand{\C}{\mathbb{C}}

\newcommand{\R}{\mathbb{R}}

\newcommand{\CP}{\mathbb{CP}}
\newcommand{\CH}{\mathbb{CH}}
\newcommand{\RH}{\mathbb{RH}}
\newcommand{\RP}{\mathbb{RP}}

\newcommand{\caA}{\mathcal{A}}

\newcommand{\caD}{\mathcal{D}}

\newcommand{\caQ}{\mathcal{Q}}

\newcommand{\so}{\mathfrak{so}}

\newcommand{\spl}{\mathfrak{sl}}
\newcommand{\su}{\mathfrak{su}}
\newcommand{\fa}{\mathfrak{a}}

\newcommand{\fg}{\mathfrak{g}}
\newcommand{\fk}{\mathfrak{k}}
\newcommand{\fm}{\mathfrak{m}}

\newcommand{\End}{\operatorname{End}}
\newcommand{\Hom}{\operatorname{Hom}}

\newcommand{\Kah}{K\" ahler\ }

\newcommand{\tr}{\operatorname{tr}}
\newcommand{\e}{\varepsilon}
\newcommand{\vol}{\operatorname{vol}}

\newcommand{\g}[2]{\langle{#1},{#2}\rangle}
\newcommand{\II}{\mathrm{I\! I}}
\newcommand{\bz}{{\bar z}}
\newcommand{\diag}{\operatorname{diag}}

\newtheorem{thm}{Theorem}[section]
\newtheorem{prop}[thm]{Proposition}
\newtheorem{lem}[thm]{Lemma}
\newtheorem{cor}[thm]{Corollary}

\theoremstyle{remark}

\newtheorem{rem}{Remark}

\numberwithin{equation}{section}

\begin{document}
\title{Cubic Differentials in the Differential Geometry of Surfaces}

\author{John Loftin and Ian McIntosh}
\thanks{
The first author's work is partially supported by a Simons Collaboration Grant for Mathematicians
210124, and by  U.S.\ National Science
Foundation grants DMS 1107452, 1107263, 1107367 ``RNMS: Geometric
structures And Representation varieties" (the GEAR Network).
} 

\address{
Department of Mathematics and Computer Science\\ Rutgers-Newark\\
Newark, NJ 07102, USA \\
email:\,{\tt loftin@rutgers.edu}
\newline\indent
Department of Mathematics\\ University of York\\
York YO10 5DD, UK \\
email:\,\tt{ian.mcintosh@york.ac.uk}
}

\maketitle

\begin{abstract}
We discuss the local differential geometry of convex affine spheres in $\re^3$ and of minimal Lagrangian surfaces in Hermitian symmetric spaces. In each case, there is a natural metric and cubic differential holomorphic with respect to the induced conformal structure: these data come from the Blaschke metric and Pick form for the affine spheres and from the induced metric and second fundamental form for the minimal Lagrangian surfaces.  The local geometry, at least for main cases of interest, induces a natural frame whose structure equations arise from the affine Toda system for $\mathfrak a^{(2)}_2$.  We also discuss the global theory and applications to representations of surface groups and to mirror symmetry.
\end{abstract}

\section{Introduction.}
Holomorphic cubic differentials naturally appear as the Pick form in the study of affine spheres in $\re^3$, and also arise from the second fundamental form for minimal Lagrangian surfaces in the symmetric spaces $\C^2$, $\CH^2$ and $\CP^2$\index{cubic differential}. In each of these cases, the natural metric (the Blaschke metric for affine spheres or the induced metric for the minimal Lagrangian surfaces) can be given conformally by $2e^{2\psi}\, |dz|^2$, where $\psi$ solves a semilinear elliptic equation of \c{T}i\c{t}eica type
\begin{equation} \label{gen-eq}
 2\psi_{z\bar z} + \epsilon |Q|^2 e^{-4\psi} +  \lambda e^{2\psi} =0,
\end{equation}
for $Q$ a holomorphic cubic differential on a simply-connected domain $\mathcal D\subset \co$. We have the following cases for the signs $\epsilon$ and $\lambda$:
$$
\renewcommand{\arraystretch}{3}
\begin{array}{|c|c|c|c|c|}
\hline
\epsilon &  \lambda & {\rm Surface} & \mbox{Group $G$} & {\rm Application}\\
\hline \hline
-1 & -1 & \mbox{minimal Lagrangian in $\CH^2$} & \mathrm{SU}(2,1)&\parbox{1in}{almost Fuchsian representations}\\[.1in] \hline
-1 & 0 &  \mbox{minimal Lagrangian in $\co^2$} &\mathrm{AU}(2) & \parbox{1in}{special Lagrangian with constant phase}\\[.1in] \hline
-1 & 1 & \mbox{minimal Lagrangian in $\CP^2$} & \mathrm{SU}(3)&\parbox{1in}{special Lagrangian cones in $\co^3$}\\[.1in] \hline
1 & -1 & \mbox{hyperbolic affine sphere}&\Sl3& \parbox{1in}{Hitchin representations}\\[.1in] \hline
1&0 & \mbox{parabolic affine sphere} & \mathrm{ASL}(2,\re)&\parbox{1in}{semi-flat Calabi-Yau on $\sph^2$}\\[.1in] \hline
1 & 1 & \mbox{elliptic affine sphere} & \Sl3&\parbox{1in}{SYZ geometry near the ``Y" vertex}\\[.1in] \hline
\end{array}
\renewcommand{\arraystretch}{1}
$$
Here $\mathrm {AU}(2)$ means the affine group with linear part $\mathrm{ U}(2)$. In each case (\ref{gen-eq}) is the zero-curvature equation for a principal $G$-bundle over the surface.  This flat bundle corresponds to a conjugacy class in the representation space ${\rm Hom}(\pi_1\Sigma,G)/G$, which is a natural generalization of the classical Teichm\"uller space when the surface $\Sigma$ has genus at least two: recall that Teichm\"uller space is a connected component of the representation space for $G=\mathrm{PSL}(2,\re)$ consisting of the Fuchsian representations.

For $\lambda=\pm1$, the structure equations for these surfaces are versions of the Toda equations for  $\fa^{(2)}_2$, the affine Dynkin diagram for $\mathrm{SL}(3,\co)$ with respect to its natural outer automorphism. The local theory of these equations goes back to \c{T}i\c{t}eica (Tzitz\'eica) \cite{tzitzeica07,tzitzeica08,tzitzeica09}, who first developed a hyperbolic analog of this system (as the structure equations for what is now known as an indefinite proper affine sphere in $\re^3$).  The remaining cases with $\lambda=0$ are completely integrable in that one can easily represent all solutions in terms of holomorphic data.

This chapter is broken into two main parts, on local and global theory.  In the local theory, the underlying integrable systems work much the same and are insensitive to changing the signs $\epsilon$ and $\lambda$. In the second part, the global theory is much more dependent on the signs, as solving (\ref{gen-eq}) on a Riemann surface often depends on the maximum principle.

The point of view we take in this chapter is to focus on global solutions of (\ref{gen-eq}) over a hyperbolic Riemann surface.
The most successful theory in this regard has been that of the hyperbolic affine spheres.  On a compact Riemann surface with background hyperbolic metric and cubic differential, there is a unique associated Blaschke metric.  Together with Cheng-Yau's powerful classification theorem for complete hyperbolic affine spheres \cite{cheng-yau77,cheng-yau86}, this gives a parametrization of the space of all convex $\rp^2$ structures on the underlying surface, and the Hitchin component of representations of the fundamental group into $\Sl3$ \cite{labourie97,labourie07,loftin01,wang91,choi-goldman93,hitchin92}.  Note the signs $\epsilon=1$ and $\lambda=-1$ are both good for using the maximum principle.

We also discuss various applications of these equations.  We discuss the geometry of minimal Lagrangian surfaces in $\CH^2$ and almost Fuchsian representations into $\mathrm{SU}(2,1)$ in Subsection \ref{ch2-subsec}, while we explain the special Lagrangian property of minimal Lagrangian surfaces in $\co^2$ in \ref{c2-subsec}.  The relationship between two-dimensional affine spheres to the Strominger-Yau-Zaslow conjecture for Calabi-Yau three-folds are recounted in \ref{slag-subsec} below, as are special Lagrangian cones in $\co^3$ derived from minimal Lagrangian surfaces in $\CP^2$. The geometry of hyperbolic affine spheres and relations to convex $\rp^2$ surfaces and Hitchin representations are covered in \ref{has-subsec}.

The equation (\ref{gen-eq}) can also be tackled using methods derived from integrable systems theory.  In the broadest sense this means local solutions can, in principle, be found using loop group factorization methods such as those described in \cite{dorfmeister-pedit-wu98}. These are not very informative in the general case, but when $Q$ is constant and $G=\mathrm{SU}(3)$, all doubly periodic solutions of (\ref{gen-eq}) can be constructed using these methods since (\ref{gen-eq}) becomes a version of the Toda field equations: see, for example, \cite{mcintosh03} and references therein. In that case, the doubly periodic solutions can be written down explicitly using abelian functions on an auxiliary Riemann surface called the \emph{spectral curve}\index{spectral curve}. Since these solutions only apply to minimal tori in $\CP^2$, we will not say much about them here.

Four of these six equations (all three affine spheres and minimal Lagrangian surfaces in $\CP^2$) are dimension reductions of special cases of special Lagrangian submanifolds of Calabi-Yau 3-folds.  The affine spheres in particular have been useful in constructing models for the Strominger-Yau-Zaslow conjecture in mirror symmetry \cite{syz,loftin04,lyz05,lyz-erratum}.

We thank the referee for useful comments.

\emph{A remark on notation.} In attempting to give a unified treatment of several different differential geometric objects on surfaces, it is perhaps inevitable that we have run into minor inconsistencies in terms of normalizing constants. In particular, we have chosen constants for the metric and cubic differential to simplify the presentation of the Toda theory, since that is a unifying principle behind this work.  These conventions differ from those the first author has made for affine spheres in the past. In particular, the presentation here of the Blaschke metric in local coordinates and of the cubic differential differ by constant factors from those in e.g. \cite{loftin01,loftin04,loftin02c, lyz05}. See Subsection \ref{two-dim} below. Moreover, we discovered a conversion error in passing between the local and global equations in \cite{loftin-mcintosh13} which is corrected here. To translate between that paper and here, the cubic differential called $U$ there is $Q/2\sqrt{2}$ in our notation here.

\section{Local theory.}

\subsection{Affine spheres} \label{affine-sph-subsec}

\subsubsection{A first example.}

\c{T}i\c{t}eica  first defined a surface closely related to our problem.  This surface is nowadays called an indefinite proper affine sphere, and involves neither a Riemann surface structure nor a cubic differential.  On the other hand, the essential integrability properties are still in place, and \c{T}i\c{t}eica's early work should be considered as the founding texts for this topic.

We very briefly recount the theory of indefinite proper affine spheres; see e.g.\ Simon-Wang \cite{simon-wang93} or Bobenko-Schief \cite{bobenko-schief99} for modern treatments. On a simply connected domain $\mathcal D\subset \re^2$, let a function $w$ satisfy \c{T}i\c{t}eica's equation
\begin{equation}
w_{xy} + e^{-2w} +Ke^w=0, \label{tzitz-eq}
\end{equation}
for $K$ a constant. Then there exists an indefinite affine sphere whose Blaschke metric is $2e^w\,dx\,dy$ with constant affine mean curvature $K$. \c{T}i\c{t}eica's equation (\ref{tzitz-eq}) is the integrability condition for the structure equations of the affine sphere.  We will not go into the geometry in this case, as we develop the case of convex affine spheres in detail below.

\subsubsection{Affine differential geometry}

Affine differential geometry is the study of the differential-geometric properties of hypersurfaces in $\re^{n+1}$ which are invariant under unimodular affine actions of $\re^{n+1}$.   These invariants are not obvious at first glance, as the usual notions of length and angle in $\re^{n+1}$ are not valid.  One way of accessing the theory is through the \emph{affine normal}\index{affine normal}, which is a third-order transverse vector field to a $C^3$ hypersurface in $\re^{n+1}$.

In order to define the affine normal, we first discuss the affine geometry a transverse vector field generates on hypersurface in $\re^{n+1}$ (see e.g.\ Nomizu-Sasaki \cite{nomizu-sasaki}).  Let $f\!: H \to \re^{n+1}$ be an immersion and $\tilde\xi\!:H\to\re^{n+1}$ a transverse vector field to $f(H)$.  For any smooth vector fields $X,Y$ on $H$, push them forward by $f_*$ and locally extend $f_*X$, $f_*Y$ and $\tilde\xi$ to smooth vector fields on $\re^{n+1}$.  Then apply the standard connection $D$ on $\re^{n+1}$ and at a point $x\in H$, use the splitting $T_{f(x)}\re^{n+1} = f_* T_xH + \langle \tilde\xi(x) \rangle$ to find
\begin{eqnarray*}
D_{f_*X}f_*Y &=& f_*(\tilde\na_X Y) + \tilde h(X,Y) \tilde \xi , \\
D_{f_*X}\tilde\xi &=& -f_*(\tilde S(X)) + \tilde \tau(X)\tilde \xi.
\end{eqnarray*}
One can easily check that $\tilde\na$ is a torsion-free connection, $\tilde h$ is a symmetric tensor, $\tilde S$ is an endomorphism of $TH$ and $\tilde\tau$ is a one-form on $H$, and that all are independent of the extensions.  Note that $D$ depends only on the affine structure on $\re^{n+1}$.  Below we suppress the $f_*$ from the notation.

Now assume $f$ is convex, or equivalently assume $\tilde h$ is definite for every transverse $\tilde\xi$. The affine normal $\xi$ can be defined by the following requirements:
\begin{itemize}
\item $h$ is positive definite.
\item $\tau = 0$.
\item For each frame of tangent vectors $X_1,\dots, X_n$, $$\det_{\re^{n+1}} (X_1,\dots,X_n,\xi)^2 = \det_{1\le i,j\le n} h(X_i,X_j).$$
\end{itemize}
The first condition means that $\xi$ points inward (so that $H$ and $\xi$ lie on the same side of the tangent plane).
Moreover, if $\tilde h$ is not definite but still nondegenerate, we can still define the affine normal up to a choice of orientation in a similar manner.

 The affine structure equations are then
\begin{eqnarray*}
D_XY &=& \na _XY + h(X,Y)\xi, \\
D_X\xi &=& -S(X).
\end{eqnarray*}
In this case, $\na$ is called the Blaschke connection\index{Blaschke connection}; $h$ is the affine second fundamental form, Blaschke metric, or affine metric\index{Blaschke metric}\index{affine metric}; and $S$ is the affine shape operator or affine third fundamental form.

Let $\hat\na$ be the Levi-Civita connection for the metric $h$. Then the \emph{cubic tensor}\index{cubic tensor}, or \emph{Pick form}\index{Pick form}, is defined to be $C=\hat\na-\na$. (This sign convention is that of \cite{loftin02c}, not that of \cite{loftin-tsui08}). In index notation, we write $C^i_{jk}$. It is immediate that $C$ is symmetric in the bottom two indices.  Moreover, $C$ satisfies the following \emph{apolarity condition}\index{apolarity condition}
\begin{lem}
$ C^i_{ij} = 0.$
 \end{lem}
\begin{proof}
This follows from the structure equations and the definition of $\xi$.  If we consider $f$ to be a coordinate immersion, let $\Gamma^k_{ij}$ denote the Christoffel symbols of $h$, and use a comma to denote covariant derivatives with respect to $h$, we find $$f_{,ij} = \frac{\partial^2 f}{\partial x^i \partial x^j} - \Gamma^k_{ij} f_{,k} = -C^k_{ij} f_{,k} + h_{ij} \xi.$$
Now compute using the definition of $\xi$:
\begin{eqnarray*}
0&=& \hat\na_j \det(f_{,1},\dots,f_{,n},\xi) \\
&=& \det (f_{,1j},\dots, f_{,n},\xi) + \cdots +\det(f_{,1},\dots,f_{,nj},\xi) + \det (f_{,1},\dots,f_{,n},\xi_{,j}) \\
&=& \det (-C^i_{1j}f_{,i},\dots, f_{,n},\xi) + \cdots + \det (f_{,1},\dots ,-C^i_{nj} f_{,i}, \xi) + 0 \\
&=& -C^i_{ij} \det(f_{,1},\dots,f_{,n},\xi).
\end{eqnarray*}
\end{proof}

\begin{lem}
$C^\ell_{ij}h_{\ell k} = C_{ijk}$ is totally symmetric.
\end{lem}
\begin{proof}
We have already seen that $C_{ijk}=C_{jik}$.  To show the remaining symmetry,
take the third covariant derivative of $f$:
\begin{eqnarray*}
f_{,ijk} &=& -C^\ell_{ij,k}f_{,\ell} - C^\ell_{ij}f_{,\ell k} + h_{ij}\xi_{,k} \\
&=& -C^\ell_{ij,k}f_{,\ell} + C^\ell_{ij} C_{\ell k}^m f_{,m} - C^\ell_{ij}h_{\ell k} \xi - h_{ij} A^\ell_k f_{,\ell}.
\end{eqnarray*}
The difference $f_{,ijk}-f_{ikj}$ involves curvature terms of the form $R^\ell_{jki}f_{,\ell}$, which lie entirely in the tangent space to $f(H)$.  Thus the part of $f_{,ijk}-f_{ikj}$ in the span of $\xi$ is 0.
\end{proof}

The other main fact about the cubic tensor is the following result of Maschke, Pick, and Berwald:
\begin{thm} \label{cubic-form-vanish}
Let $H$ be a smooth hypersurface in $\re^{n+1}$ whose second fundamental form is nondegenerate.  Then $H$ is a subset of a hyperquadric if and only if $C_{ij}^k =0$ identically.
\end{thm}

\subsubsection{Affine spheres}

Now that the basic affine differential invariants have been introduced, we can define the affine spheres.  Recall that one of the first invariants in the Euclidean geometry of hypersurfaces is the unit normal vector.  A Euclidean sphere can then be defined as a hypersurface all of whose normal lines pass through a single point, the center.  An \emph{affine sphere}\index{affine sphere} can be defined in the same way:  A \emph
{proper affine sphere}\index{proper affine sphere}\index{affine sphere!proper} is a hypersurface in $\re^{n+1}$ all of whose affine normal lines pass through a single point in $\re^{n+1}$, called the center.  If the affine sphere is locally convex, the proper affine spheres split into two types: \emph{hyperbolic affine spheres}\index{hyperbolic affine sphere}\index{affine sphere!hyperbolic}, all of whose affine normals point away from the center, and \emph{elliptic affine spheres}\index{elliptic affine sphere}\index{affine sphere!elliptic}, all of whose affine normals point toward the center.  We also define an \emph{improper affine sphere}\index{improper affine sphere}\index{affine sphere!improper} as a hypersurface in $\re^{n+1}$ all of whose affine normal lines are parallel.  An improper affine sphere which is locally convex is called a \emph{parabolic affine sphere}\index{parabolic affine sphere}\index{affine sphere!parabolic}.

It is not easy to write down affine spheres except in highly symmetric examples.  As mentioned above, the easiest examples are hyperquadrics in $\re^{n+1}$.  In particular, ellipsoids are elliptic affine spheres: a round sphere in $\re^{n+1}$ is clearly an elliptic affine sphere by symmetry, and ellipsoids are affine images of round spheres.  Elliptic paraboloids are parabolic affine spheres, and a sheet of an elliptic hyperboloid is a hyperbolic affine sphere.

There is another prominent homogeneous example of hyperbolic affine spheres which goes back to \c{T}i\c{t}eica in dimension 2 \cite{tzitzeica08}, and is due to Calabi in higher dimensions \cite{calabi72}. In the first orthant $\{x^i>0\}\subset \re^{n+1}$, the level set $L=\{x^i>0, \Pi_i x^i = 1\}$ is a hyperbolic affine sphere.

The structure equations for affine spheres become
\begin{eqnarray*}
D_XY &=& \na_XY + h(X,Y)\xi, \\
D_X\xi &=& -\lambda \,X,
\end{eqnarray*}
for $\lambda$ a real constant.  An affine sphere is improper if and only if $\lambda=0$.  Elliptic affine spheres have $\lambda>0$, while hyperbolic affine spheres have $\lambda<0$.  By scaling, we assume $\lambda \in \{-1,0,1\}$. It is also useful to translate so that proper affine spheres have their centers at the origin. In this case $\xi = -\lambda f$ for $f$ the immersion.  For improper affine spheres, we may apply a linear map to ensure that $\xi$ is the last coordinate vector.

\subsubsection{Monge-Amp\`ere equations}

Affine spheres can be parametrized by solutions to  Monge-Amp\`ere equations\index{Monge-Amp\`ere equation}.  In particular, we find an equation for a proper affine sphere with center at the origin and $\lambda=\pm1$.  Let $\Omega\subset\re^n$ be a domain, where $\re^n$ is considered as an affine subspace $\re^n\times\{1\}\subset\re^{n+1}$, and let $u\!:\Omega\to\re$ be a function, and consider the radial graph of $\frac\lambda u$
$$ \left\{ f(t)= \frac\lambda {u(t)}\, (t,1) : t\in\Omega\right\}.$$
Compute using $\xi = -\lambda f$ that
\begin{eqnarray*}
\frac{\partial f}{\partial t^i} &=& -\frac{u_i}u\,f + \frac\lambda u (e_i,0), \\
\frac{\partial^2f}{\partial t^i \partial t^j} &=& -\frac{u_{ij}}u\,f -\frac{u_i}u \,\frac{\partial f}{\partial t^j} - \frac{u_j}u \, \frac{\partial f}{\partial t^i}, \\
h_{ij}&=& \lambda\,\frac{u_{ij}}u,\\
\det\left(\frac{\partial f}{\partial t^1},\dots,\frac{\partial f}{\partial t^n},\xi\right) &=&- \frac{\lambda^n}{u^{n+1}}.
\end{eqnarray*}
So the equation for $f$ to be an immersion of a proper affine sphere is $$\det u_{ij} = \left(\frac \lambda u \right)^{n+2}.$$

The story for parabolic affine spheres is similar.  We set $\xi = e_{n+1}$, the last coordinate vector, and consider the ordinary Cartesian graph $f(t)=(t,u(t))$ for $t\in\Omega\subset\re^n$.  Then it is straightforward to check that $f$ is an immersion of a parabolic affine sphere if and only if $\det u_{ij} = 1$.

\subsubsection{Two-dimensional affine spheres} \label{two-dim}

For convex affine spheres, the affine metric is positive-definite, and thus it provides a conformal structure in dimension two.  In this case, the cubic tensor can be identified with a holomorphic cubic differential \cite{wang91,simon-wang93}.  We now derive the structure equations for these affine spheres.

Choose $z=x+iy$ a local conformal coordinate with respect to $h$, so that $h=2e^{2\psi}|dz|^2$.
(This convention for the metric is not the same as that in \cite{loftin01,loftin02c,loftin06} -- where $h=e^\psi|dz|^2$ is used instead; the present convention is more convenient for the Toda theory.) Parametrize the affine sphere by $f\!: \mathcal D \to \re^3$, where $\mathcal D$ is a simply connected domain in $\co$. Since $\{\frac1{\sqrt2}e^{-\psi}f_x, \frac1{\sqrt2}e^{-\psi}f_y\}$ is an orthonormal basis of the tangent space, the affine normal satisfies
$$\det (e^{- \psi}f_x, e^{-\psi}f_y, \xi) = 2,$$
which implies
$$ \det (f_z,f_{\bar z},\xi) =  i e^{2\psi}.$$
The affine structure equations become
\begin{eqnarray*}
D_XY &=& \na                                                                                             _XY + h(X,Y)\xi, \\
D_X\xi &=& -\lambda\, X.
\end{eqnarray*}
Consider the coordinate frame $\{e_1 =f_z = f_*(\frac\partial{\partial z}), e_{\bar 1} = f_{\bar z} = f_*(\frac\partial{\partial \bar z})\}$. Then
$$ h(f_z,f_z)= h(f_{\bar z}, f_{\bar z}) = 0,\quad h(f_z,f_{\bar z}) =  e^{2\psi}.$$
Let  $\theta$ be the matrix of connection one-forms for $\na$
$$ \na e_i = \theta^j_ie_j, \quad i,j\in\{1,\bar 1\}.$$
If $\hat\theta$ is the matrix of connection one-forms of the Levi-Civita connection, then $$ \hat\theta^1_{\bar 1} = \hat \theta^{\bar 1}_1 = 0, \quad \hat\theta^1_1 = \partial \psi, \quad \hat \theta^{\bar 1}_{\bar 1} = \bar\partial \psi.$$

The difference $\hat\na - \na$ is the cubic tensor.
$$ \hat\theta^j_\ell - \theta^j_\ell = C^j_{\ell k} \rho^k,$$
where $\{\rho^1 = dz, \rho^{\bar 1} = d\bar z\}$ is the dual frame of one-forms.
The apolarity condition is then
$$ C^1_{1k} + C^{\bar 1}_{\bar 1 k} = 0, \quad k \in \{1,\bar 1\},$$
which, upon lowering the indices, implies
$$ C_{\bar 1 1 k} + C_{1\bar 1 k} = 0.$$
Since the cubic tensor is totally symmetric, this implies all components of $C$ must vanish except $C_{111}$ and $C_{\bar 1\bar 1\bar 1}=\overline{C_{111}}$. This discussion completely determines $\theta$:
$$\left(\begin{array}{cc}\theta^1_1 & \theta^1_{\bar 1} \\[1mm] \theta^{\bar 1}_1 & \theta^{\bar 1}_{\bar 1} \end{array}\right) =
\left(\begin{array}{cc} 2\partial \psi & \bar Q e^{-2\psi}d\bar z \\ Q e^{-2\psi} dz & 2\bar\partial \psi \end{array} \right)$$
for $Q= C^{\bar 1}_{11} e^{2\psi}$ (this convention for the cubic differential is half the corresponding quantity $U$ in \cite{loftin01,loftin02c,loftin06}).

Since $D$ is the standard connection on $\re^3$, we have for example $D_{f_z}f_z = D_{\frac\partial{\partial z}} f_z = f_{zz}.$  The structure equations then become
\begin{eqnarray*}
f_{zz} &=& 2\psi_z f_z + Q e^{-2\psi} f_{\bar z}, \\
f_{\bar z \bar z} &=& \bar Q e^{-2\psi} f_z + 2\psi_{\bar z} f_{\bar z}, \\
f_{z\bar z} &=&  e^{2\psi} \xi.
\end{eqnarray*}
Then, together with the first-order equations $\xi_z = -\lambda f_z$, $\xi_{\bar z} = -\lambda f_{\bar z}$, we get a first-order linear system in  $f_z$, $f_{\bar z}$ and $\xi$:
\begin{eqnarray*}
\frac\partial{\partial z}
\left(\begin{array}{c}  f_z \\ f_{\bar z} \\ \xi \end{array} \right)
&=&
\left(\begin{array}{ccc}  2\psi_z & Qe^{-2\psi} &0 \\
 0 & 0 & e^{2\psi}  \\ -\lambda & 0 & 0 \end{array}\right)
\left(\begin{array}{c}  f_z \\ f_{\bar z} \\ \xi \end{array} \right) , \\
\frac\partial{\partial \bar z}
\left(\begin{array}{c}  f_z \\ f_{\bar z} \\ \xi \end{array} \right)
&=&
\left(\begin{array}{ccc}  0 & 0 & e^{2\psi}\\
 \bar Q e^{-2\psi} & 2\psi_{\bar z} & 0 \\ 0 & -\lambda & 0 \end{array}\right)
\left(\begin{array}{c}  f_z \\ f_{\bar z} \\ \xi \end{array} \right) .
\end{eqnarray*}

Given initial conditions for the frame $\{f_z,f_{\bar z},\xi\}$ at $z_0\in\mathcal D$, the linear system has a unique solution as long as the mixed partials commute  (this can be traced back to the Frobenius Theorem).  In other words, we require $(f_{zz})_{\bar z} = (f_{z\bar z})_z$ and $(f_{\bar z \bar z})_z = (f_{z\bar z})_{\bar z}.$  This becomes
\begin{eqnarray}
\label{aff-sph-eq-local}
0&=& 2\psi_{z\bar z} + |Q|^2 e^{-4\psi} +  \lambda e^{2\psi},\\
\nonumber
0&=& Q_{\bar z}.
\end{eqnarray}
It is also easy to check that $Q$ transforms as a cubic differential under holomorphic coordinate changes. Altogether, we have shown:
\begin{thm} \label{aff-sph-integrate}
Fix $\lambda\in\{-1,0,1\}$.
Let $\mathcal D\subset \co$ be a simply-connected domain, $Q$ is a holomorphic cubic differential on $\mathcal D$, $\psi\!:\mathcal D\to \re$ satisfies (\ref{aff-sph-eq-local}) $z_0\in\mathcal D$, and $\xi_0,f_0\in\re^3$, $a\in\co^3$ so that $\det(a,\bar a, \xi_0) = ie^{2\psi(z_0)}$. Then there is a unique immersion of an affine sphere $f\!:\mathcal D \to \re^3$ so that $$ f(z_0)=f_0, \quad \xi(z_0)=\xi_0, \quad
f_z(z_0) = a, \quad f_{\bar z}(z_0) = \bar a,$$ the pullback under $f$ of the Blaschke metric and cubic form are $e^{2\psi}|dz|^2$ and $Q\,dz^3$ respectively.  The affine sphere is hyperbolic, parabolic, or elliptic if $\lambda=-1,0,1$ respectively.
\end{thm}

We typically assume for parabolic affine spheres ($\lambda=0$) that $\xi_0=e_3$, which implies $\xi=e_3$.  For hyperbolic ($\lambda=-1$) and elliptic ($\lambda=1$) affine spheres, we assume $f_0=-\lambda \xi_0$, which implies $f=-\lambda \xi$.

\subsubsection{Dual affine spheres}

For each of these definite affine spheres, there is a dual affine sphere of the same type which is related to the Legendre transform.  Given a smooth convex function $s=s(x^1,\dots, x^n)$,  the \emph{Legendre transform}\index{Legendre transform} $s^*$ of $s$ is given by the formula
$$s^* + s = x^i \,\frac{\partial s}{\partial x^i}.$$
With respect to the variables $y_i = \frac{\partial s}{\partial x^i}$, $s^*$ is a convex function. The duality extends to the Monge-Amp\`ere equation
$$\det \frac{\partial^2s}{\partial x^i \partial x^j} =1 \qquad \Longleftrightarrow \qquad \det\frac{\partial^2s^*}{\partial y_i \partial y_j}=1.$$
Therefore, the Legendre transform takes a parabolic affine sphere given by the graph $\{(x,s(x)\}$ where $x\in \re^n$ to a parabolic affine sphere in the dual space $\{y,s^*(y)\}$, where $y\in\re_n$ and $\re_n$ is the dual vector space to $\re^n$.

For proper affine spheres, the duality is through the \emph{conormal map}\index{conormal map}.  For a hypersurface $L\subset \re^{n+1}$ transverse to the position vector, the conormal map $N\!:L\to \re_{n+1}$ is given by $N(x) = \ell$, where $\ell\!: x\mapsto 1$ and $\ell\!:T_xL\to 0$. The conormal map is naturally related to the Legendre transform (see e.g. \cite{loftin10}), and hyperbolic and elliptic affine spheres are taken by the conormal map to affine spheres of the same type.

The duality in each of these cases is an isometry of the Blaschke metric, and takes the cubic form $C\mapsto -C$.

\subsection{Harmonic maps and minimal surfaces.}\index{harmonic map}\index{minimal surface}
There is an intimate link between harmonic maps of surfaces and holomorphic differentials, which has its origin
in the notion of the \emph
{Hopf differential}\index{Hopf differential} of a surface in Euclidean $\R^3$. The Hopf differential is a
complex quadratic differential (of type $(2,0)$) on the surface built from its second fundamental form, and it is
holomorphic precisely when the surface has constant mean curvature
(equally, by the Ruh-Vilms theorem, when its Gauss map is a harmonic map). In that case the Gauss map induces a
closely related holomorphic quadratic differential, a scalar multiple of the Hopf differential, which vanishes
precisely when the mean curvature is zero (i.e., when the surface is minimal) or when the surface is totally umbilic.
This is the simplest example of a more general theory for harmonic maps of surfaces which explains the
appearence of quadratic, cubic, and higher order, holomorphic differentials.

To summarize this theory, let $(\Sigma,\gamma)$ be a closed Riemannian surface and $(N,g)$ a Riemannian manifold
of dimension $n$. The harmonic map equations are the condition that a $C^2$ map $f:\Sigma\to N$ is a
critical point of the \emph{Dirichlet energy}\index{Dirichlet energy}
\begin{equation}\label{eq:Dirichlet}
E(f) = \tfrac12\int_\Sigma\|df\|^2\vol_\gamma,\quad \|df\|^2 = \tr_\gamma f^*g.
\end{equation}
Let us restrict our attention to the case of immersions.
This functional is closely related to the area $A(f)$ of $\Sigma$ in the induced metric $f^*g$: when $f$ is
an isometric immersion, $\gamma=f^*g$ and $E(f)=A(f)$. Indeed, in this case $A(f)$ and $E(f)$ satisfy the
same Euler-Lagrange equations (although this does not follow simply from $E(f)=A(f)$, since the variations for
the former must preserve the metric $\gamma$, while those for the latter do not). The cleanest way to think of
the Euler-Lagrange equations is via the second fundamental form of the map $f$, which is the symmetric tensor
\[
\nabla df(X,Y) = \nabla^g_Xf_*Y - f_*(\nabla^\gamma_XY),\quad X,Y\in\Gamma(TM).
\]
This takes values in $f^{-1}TN$.  It generalizes the usual notion of the second fundamental form in submanifold
theory, with which it agrees when $\gamma=f^*g$. Taking its trace with respect to
$\gamma$ gives the
\emph{tension field}\index{tension field} $\tau(f) = \tr_\gamma\nabla df$: this agrees with the mean curvature of $f$ when $\gamma$ is
the induced metric. Given a smooth compactly supported vector field $V\in \Gamma(f^{-1}TN)$, the first variation is
$$\delta E_V(f) = -\int_\Sigma g(V,\tau(f))\,{\rm vol}_\gamma.$$
The map $f$ is harmonic when $\tau(f)=0$. This is a (typically nonlinear)
generalization of the Laplace equation (with which it agrees when $f:\Sigma\to\R$): it is second order, quasi-linear
and elliptic.

When $\Sigma$ is a surface, it is easy to show that
$E(f)$ is invariant under conformal changes of $\gamma$, so the Euler-Lagrange equations depend only on the
complex structure $\Sigma$ obtains from its metric $\gamma$. Further, on a surface every metric is locally
conformally flat, i.e., there exist about every point local coordinates $(x,y)$ for which $\gamma =
\sigma(dx^2+dy^2)$ for some positive function $\sigma$. Thus $z=x+iy$ is a local complex coordinate on
$\Sigma$. The next result shows that the harmonic map equations for $f:\Sigma\to N$ are a form of Cauchy-Riemann
equation for $df(\partial/\partial z)$.
\begin{lem}
Let $f:(\Sigma,\gamma)\to (N,g)$ be a $C^2$ map from a surface, and let $z$ be a local complex coordinate
with $\gamma = \sigma|dz|^2$. Then
\begin{equation}
\tau(f) = 4\sigma^{-1}\nabla^g_{\bar Z}f_*Z,
\end{equation}
where $Z = \partial/\partial z$. In particular, $f$ is harmonic if and only if $\nabla^g_{\bar Z}f_*Z=0$.
\end{lem}

\subsubsection{Quadratic differentials.}\index{quadratic differential}
The eigenspaces of the complex structure on $\Sigma$ give the type decomposition of $T\Sigma^\C$ into
$T^{1,0}\Sigma\oplus T^{0,1}\Sigma$. When we extend $f^*g$ complex-linearly to $T\Sigma^\C$ it has a type
decomposition:
\[
f^*g = (f^*g)^{2,0} + (f^*g)^{1,1} + (f^*g)^{0,2}.
\]
The map $f$ will be \emph{weakly conformal}\index{weakly conformal} (i.e., $f^*g =s\gamma$ for some non-negative function $s$)
precisely when $f^*g$ has type $(1,1)$, i.e, when $(f^*g)^{2,0}=0$. By the previous lemma and the remarks
above it, $f(\Sigma)$ will be a minimal surface
(i.e., will have vanishing mean curvature) whenever it is a conformal harmonic map. When $f$ is only weakly conformal,
we will say it is a \emph{branched minimal immersion}\index{branched minimal immersion}. We are now in a position to see the source of the quadratic
holomorphic differentials in harmonic surface theory.
\begin{cor}
Let $f:(\Sigma,\gamma)\to (N,g)$ be a harmonic immersion of a surface, then $f^*g^{2,0}$ is a holomorphic
quadratic differential on $\Sigma$. In particular, every harmonic $2$-sphere is a (possibly branched) minimal $2$-sphere.
\end{cor}
The proof is short and makes perfectly
transparent how holomorphicity follows from the previous lemma.
\begin{proof}
In a local conformal coordinate $z$ on $\Sigma$ we note that, for $Z=\partial/\partial z$,
\[
(f^*g)^{2,0} = g(f_*Z,f_*Z)dz^2.
\]
 Now
\[
\bar Z g(f_*Z,f_*Z) = 2 g(\nabla^g_{\bar Z} f_*Z,f_*Z),
\]
which vanishes when $f$ is harmonic. Since there are no non-zero holomorphic differentials on the Riemann
sphere, every harmonic map of the sphere must be weakly conformal.
\end{proof}
For an immersed surface $\varphi:\Sigma\to \R^3$ with Gauss map $f:\Sigma\to (S^2,g)$ and second fundamental form $\II$,
this differential $(f^*g)^{2,0}$ is related to the Hopf differential $\II^{2,0}$ of $\varphi$ by
\[
(f^*g)^{2,0} = \tfrac{1}{\sqrt{2}}H\II^{2,0},
\]
where $H=\tfrac12\tr\II$ is the mean curvature. In particular, $H$ is constant when $f$ is harmonic and in that case
we may deduce (when $H\neq 0$) that the Hopf differential is also holomorphic.

\subsubsection{Cubic differentials.}
To describe the appearance of holomorphic cubic differentials we will restrict our attention to the case where
$(N,g,J)$ is a K\"ahler manifold with Hermitian metric $h(X,Y)=g(X,Y)-ig(JX,Y)$. In that case we have a type
decomposition for both $T\Sigma^\C$ and $TN^\C$. For $f:\Sigma\to N$ define
\[
\partial f:T\Sigma^\C\to T^{1,0}N;\quad \partial f(Z) = df(Z)^{1,0} = \tfrac12(df(Z)-iJdf(Z)).
\]
Since $TN\simeq T^{1,0}N$ we can view $h$ as a Hermitian inner product on this bundle, with corresponding
connection on $f^{-1}T^{1,0}N$. The harmonic map equations for $f$ can be written
$\nabla_{\bar Z}\partial f(Z)=0$.  Now define a quadratic and a cubic differential on $\Sigma$ by
\[
Q_2 = h(\partial f(Z),\partial f(\bar Z))dz^2,\quad Q_3 = h(\nabla_Z\partial f(Z),\partial f(\bar Z))dz^3,
\]
where $Z=\partial/\partial z$ for a local complex coordinate $z$ on $\Sigma$. Although this definition is local,
it does indeed extend globally.
Notice that $Q_2$ is, up to a factor of $2$, $f^*g^{2,0}$. So it vanishes when $f$ is a branched minimal
immersion.
\begin{thm}
When $f:\Sigma\to N$ is a branched minimal immersion into a K\" ahler manifold of constant holomorphic sectional
curvature, $Q_3$ is a holomorphic cubic differential.
\end{thm}
\begin{rem}
As stated this result is due to Wood \cite{wood84}:
$Q_2$ and $Q_3$ are the first two in a sequence of differentials $Q_k$ obtained by taking higher
covariant derivatives of $\partial f$, and $Q_k$ is holomorphic when $Q_j=0$ for all $j<k$.
This was known earlier
for maps into $\CP^n$ by Eells \& Wood \cite{eells-wood83} and Chern \& Wolfson \cite{chern-wolfson83}.
The idea has its origin in the work of Calabi \cite{calabi67} on the construction of holomorphic differentials
related to minimal surfaces in $S^n$, but in that case the differentials are all of even degree, and so no
cubic differentials arise. Burstall \cite{burstall95} gives an exposition of this which contrasts the case of
maps into spheres with maps into complex projective spaces.
\end{rem}

\subsection{Minimal Lagrangian surfaces in K\" ahler $4$-folds}

In the case where $(N,g,J)$ is a K\" ahler manifold with $n=4$, with K\" ahler  form $\omega = g(J\ ,\ )$, one can
look in particular at Lagrangian immersions $f:\Sigma\to N$, i.e., look at the condition that $f^*\omega=0$ (equally,
$JT\Sigma = T\Sigma^\perp$ in $f^{-1}TN$).
It is well-known that in this situation the cubic tensor
\[
C(X,Y,W) = g(\II_f(X,Y),Jf_*W) = -\omega(\II_f(X,Y),f_*W),\quad X,Y,W\in\Gamma(T\Sigma),
\]
is totally symmetric. This carries all the information of the second fundamental form $\II_f$, since $\II_f$ takes values
in the normal bundle. When we extend $C$ complex-multilinearly to $T^{1,0}\Sigma$ we obtain a cubic differential $C^{3,0}$.
The next result shows that when $f$ is minimal this essentially equals $Q_3$.
\begin{lem}
Let $f:\Sigma\to N$ be a branched minimal Lagrangian immersion into a K\" ahler $4$-fold. Then $C^{3,0}=iQ_3$, and
the shape operator $\caA_f:T\Sigma^\perp\to\End(T\Sigma)$ has norm
$\|\caA_f(\xi)\|=\|Q_3\|/\sqrt{2}$ for any unit normal vector field  $\xi\in\Gamma(T\Sigma^\perp)$.
\end{lem}
\begin{proof}
Let $X = \partial_x$, $Y=\partial_y$, $Z=\tfrac12(X-iY)$ for a local complex $z=x+iy$ on $\Sigma$. Then
\begin{equation}\label{eq:caQ}
Q_3 = \caQ dz^3,\quad \caQ = h(\nabla_Z\partial f(Z),\partial f(\bar Z)).
\end{equation}
Write the metric $\gamma = f^*g$ in local coordinates as $\gamma=\sigma |dz|^2$. Then
\[
\|Q_3\|^2= \frac{8|\caQ|^2}{\sigma^3}.
\]
To compute \eqref{eq:caQ} we recall that the expression for $h$ on $T^{1,0}N$ is related to its definition on $TN$
by
\[
h(V,W) = h(V+\bar V,W+\bar W),\quad V,W\in T^{1,0}N.
\]
Now, since $\nabla$ commutes with $J$,
\[
\nabla_Z\partial f(Z) = (\nabla_Z f_*Z)^{1,0},\quad \partial f(\bar Z) = (f_*\bar Z)^{1,0}
\]
and for any $A,B\in TN$ it is easy to check that
\[
(A+iB)^{1,0} +\overline{(A+iB)^{1,0}} = A+JB.
\]
Therefore,
\[
\nabla_Z\partial f(Z) + \overline{\nabla_Z\partial f(Z)} =
\tfrac14 (\nabla_Xf_*X-\nabla_Yf_*Y - J\nabla_Yf_*X - J\nabla_Xf_*Y)
\]
and
\[
\partial f(\bar Z) + \overline{\partial f(\bar Z)} = \tfrac12 (f_*X+Jf_*Y).
\]
Thus from \eqref{eq:caQ} we compute
\begin{multline}\label{eq:8Q}
8\caQ = h(\nabla_Xf_*X-\nabla_Yf_*Y,f_*X) - 2h(\nabla_Yf_*X,f_*Y)\\
 - i\left[h(\nabla_Xf_*X-\nabla_Yf_*Y,f_*Y)+ 2h(\nabla_Xf_*Y,f_*X)\right],
\end{multline}
Since $f$ is conformal and Lagrangian we have $h(f_*X,f_*Y)=0$,
and so
\[
h(\nabla_Yf_*X,f_*Y) = -h(f_*X,\nabla_Yf_*Y) ,\quad h(\nabla_Xf_*Y,f_*X) = -h(f_*Y,\nabla_Xf_*X).
\]
Further, in conformally flat coordinates the harmonic map equations are
\[
\nabla_Xf_*X+\nabla_Yf_*Y = 0.
\]
Applying these identities to \eqref{eq:8Q} we obtain
\begin{eqnarray*}
2\caQ & = & -\omega(\nabla_Xf_*X,f_*Y) +i\omega(\nabla_Yf_*Y,f_*X)\\
&=& -\omega(\II_f(X,X),f_*Y) + i \omega(\II_f(Y,Y),f_*X).
\end{eqnarray*}
On the other hand, a very similar calculation shows that
\[
2C(Z,Z,Z) = -\omega(\II_f(Y,Y),f_*X) -i\omega(\II_f(X,X),f_*Y) = 2i\caQ.
\]
Now recall that the shape operator for $f$ is defined by
\[
g(\caA_f(\xi)X,Y) = g(\xi,\II_f(X,Y)),\quad \xi\in\Gamma(T\Sigma^\perp),\ X,Y\in\Gamma(T\Sigma).
\]
By normalizing the local frame $X,Y$ for $T\Sigma$, and using the fact that
$\caA_f(\xi)$ is symmetric and trace-free, we obtain
\[
\|\caA_f(\xi)\|^2 = \sigma^{-2}[g(\caA_f(\xi) X,X)^2 + g(\caA_f(\xi) X,Y)^2].
\]
When we write $\xi = \alpha Jf_*X+\beta Jf_*Y$, with $\alpha^2+\beta^2=\sigma^{-1}$, we get
\begin{eqnarray*}
g(\xi,\II(X,X))& =& -\alpha\omega(\nabla_Xf_*X,f_*X)-\beta\omega(\nabla_Xf_*X,f_*Y)\\
&=& 2\alpha \caQ_I +2\beta \caQ_R,
\end{eqnarray*}
for $\caQ = \caQ_R+i\caQ_I$, while
\begin{eqnarray*}
g(\xi,\II(X,Y))& =& -\alpha\omega(\nabla_Xf_*Y,f_*X)-\beta\omega(\nabla_Xf_*Y,f_*Y)\\
&=& 2\alpha \caQ_R -2\beta \caQ_I.
\end{eqnarray*}
It follows that
\[
\|\caA_\xi\|^2  =  \sigma^{-2}(\alpha^2+\beta^2)4(\caQ_R^2+\caQ_I^2) = 4|\caQ|^2\sigma^{-3} = \tfrac12 \|Q_3\|^2.
\]

\end{proof}

\subsection{Minimal Lagrangian surfaces in $\C^2$, $\CP^2$ and $\CH^2$}

The observations in the previous section are particularly useful in the study of minimal Lagrangian surfaces in the three
model spaces for \Kah $4$-folds of constant holomorphic section curvature: $\C^2$, $\CP^2$ and $\CH^2$.
We will outline here why such surfaces are uniquely determined, up to isometries of the ambient space,
by their induced metric and the cubic holomorphic differential $Q_3$: this is based on the exposition in \cite{loftin-mcintosh13}, where
only the case of $\CH^2$ is treated but it is straightforward to adapt this to $\CP^2$ by a simple change of sign.
We will summarize the derivation of the equations
governing such surfaces in such a way that their link to a certain version of the \emph{Toda lattice equations}\index{Toda lattice equations} becomes
apparent, and then comment further on this link below.

We start with the case of $\co^2$, although it is not governed by the Toda lattice equations, since the derivation is particularly simple in this case.
On $\co^2$, consider the Hermitian inner product $$\langle v,w \rangle = v_1\bar w_1 + v_2 \bar w_2$$
and metric and symplectic form given by
$$ \langle v,w\rangle = g(v,w) - i\,\omega(v,w).$$
Now let $f \!: \mathcal D\to \co^2$, where $\mathcal D$ is a simply-connected domain in $\co$, and $f$ will be our minimal Lagrangian immersion.  Assume that $f^*g$ is conformal to the standard metric on $\co$.  We have the following conditions for $f$ to be minimal Lagrangian:
\begin{eqnarray*}
\mbox{conformal} &\quad\Longleftrightarrow\quad & \langle f_z, f_{\bar z} \rangle = 0, \\
\mbox{Lagrangian} &\quad\Longleftrightarrow\quad & \langle f_z,f_z\rangle = \langle f_{\bar z}, f_{\bar z} \rangle , \\
\mbox{harmonic} &\quad\Longleftrightarrow\quad & f_{z\bar z} = 0.
\end{eqnarray*}
Set $e^{2\psi} = \langle f_z,f_z\rangle = \langle f_{\bar z}, f_{\bar z} \rangle$ so that $f^*g = 2e^{2\psi}|dz|^2$.

Now we determine the second derivatives of $f$ in terms of its first derivatives.  Differentiate $\langle f_z, f_z \rangle = e^{2\psi}$ by $z$ to find
$$ \langle f_{zz}, f_z\rangle + \langle f_z,f_{z\bar z} \rangle = 2e^{2\psi}\psi_z.$$
But now $f_{z\bar z}=0$ shows $\langle f_{zz} , f_z\rangle = 2e^{2\psi}\psi_z$.  We also define $Q = -\langle f_{zz}, f_{\bar z} \rangle$, and we may again use the harmonicity of $f$ to show $Q_{\bar z}=0$. Since $f_z$ and $f_{\bar z}$ are orthogonal vectors of norm $e^\psi$, we  find
$$f_{zz}  = 2\psi_z f_z -Q e^{-2\psi}f_{\bar z}.$$
We also compute $0 = \langle f_z,f_{\bar z}\rangle_z = \langle f_{zz}, f_{\bar z} \rangle + \langle f_z, f_{\bar z \bar z} \rangle$ to show that $\langle f_z, f_{\bar z \bar z} \rangle = Q$ and so $\langle f_{\bar z \bar z}, f_z \rangle = \bar Q$.  Also compute $\langle f_{\bar z}, f_{\bar z} \rangle_{\bar z} = 2e^{2\psi}\psi_{\bar z}$ to show
$$f_{\bar z \bar z} = \bar Q e^{-2\psi}f_z + 2 \psi_{\bar z} f_{\bar z}.$$

Again, one can check that the integrability conditions boil down to $(f_{zz})_{\bar z} = (f_{z\bar z})_z = 0$.  Compute
\begin{eqnarray*}
0=(f_{zz})_{\bar z} &=& 2\psi_{z\bar z} f_z + 2\psi_z f_{z\bar z} - Q_{\bar z} e^{-2\psi}f_{\bar z} + 2Q e^{-2\psi}\psi_{\bar z} f_{\bar z} - Qe^{-2\psi}f_{\bar z \bar z} \\
&=& 2\psi_{z\bar z} f_z -|Q|^2 e^{-4\psi} f_z .
\end{eqnarray*}
Thus we have proved
\begin{thm}
A conformal immersion $f\!:\mathcal D\to \co^2$ is minimal Lagrangian if and only if for $f^*g = 2e^{2\psi}|dz|^2$ and $Q = -\langle f_{zz},f_{\bar z} \rangle$, we have
$$2\psi_{z\bar z} = |Q|^2 e^{-4\psi}.$$
  Moreover, let $Q\!:\mathcal D \to \co$ be holomorphic and $\psi\!:\mathcal D \to \re$ satisfy $2\psi_{z\bar z} = |Q|^2 e^{-4\psi}$.  Then for any $z_0\in \mathcal D$ and $f_0,p,q\in\co^2$ so that $\langle p,q\rangle = 0$ and $\langle p,p\rangle = \langle q,q\rangle = e^{2\psi(z_0)}$, there is a unique minimal Lagrangian immersion $f\!: \mathcal D\to \co^2$ so that $f(z_0) = f_0$, $f_z(z_0) = p$, $f_{\bar z}(z_0) = q$, $f^*g = 2e^{2\psi}|dz|^2$, and $Q = -\langle f_{zz},f_{\bar z} \rangle$.
\end{thm}

Now let $N$ stand for either $\CP^2$ or $\CH^2$, each with its
Hermitian metric of constant holomorphic sectional curvature, which we normalize to $\pm 4$. We will view $N$ as a
Hermitian symmetric space: $N\simeq G/K$, where $G$ is $U(3)$ for $\CP^2$ and $U(2,1)$ for $\CH^2$. Each manifold can be
modeled using projective geometry. To treat these simultaneously, equip $\C^3$ with one or other of the
Hermitian forms
\[
\g{v}{w}_\pm = v_1\bar w_1 + v_2\bar w_2 \pm v_3\bar w_3.
\]
While $\CP^2$ is the space of all complex lines in $\C^3$, $\CH^2$ is the space of all complex lines in
\[
W_- = \{w\in\C^3:\g{w}{w}_-<0\}.
\]
In either case, for a line $[w]\in N$, the form $\g{\ }{\ }_{\pm}$ is positive-definite on its perpendicular $[w]^\perp$.
Let $L\subset N\times\C^3$ denote the tautological subbundle over the space of lines $N$, then we have the standard
identification
\[
T^{1,0}N\simeq \Hom(L,L^\perp) \subset N\times\End(\C^3);\quad Z\mapsto \pi_L^\perp Z,
\]
where $\pi_L:\C^3\to L$ is the orthogonal projection for the Hermitian form and we are thinking of $Z$ as a derivation on
local sections of $L$. In this model, the Hermitian metric on $N$ can be expressed as
\[
h(Z,W) = \g{\pi_L^\perp Z\sigma}{\pi_L^\perp W\sigma}_{\pm},\quad\ \text{whenever}\ |\sigma|_{\pm}=\pm 1.
\]
It is easy to show that $G$ acts transitively and isometrically on $N$, using the standard action of $\mathrm{GL}(3,\C)$ on
projective $3$-space, and that the line $[e_3]$ generated by $e_3=(0,0,1)$ is a point on $N$ in this model. We will take
$K$ to be the isotropy group of this point.

Now suppose $\caD\subset\C$ is an open $1$-connected domain with complex coordinate $z=x+iy$ and $f:\caD\to N$ is a
minimal Lagrangian immersion. The Lagrangian condition means that the pullback $f^*L$ of the tautological bundle has a
smooth section $\varphi$ with $|\varphi|_{\pm}=\pm 1$ which is horizontal, i.e., $\g{\varphi}{d\varphi}_\pm=0$. Combining this with
the fact that $f$ must be a conformal immersion ensures that the triple
\[
f_1 = \frac{\varphi_z}{|\varphi_z|_\pm},\quad f_2 = \frac{\varphi_{\bar z}}{|\varphi_{\bar z}|_\pm},\quad f_3 = \varphi,
\]
forms a $\g{\ }{\ }_\pm$-unitary frame for $f$, i.e., the matrix $F$ with those
columns gives a map $F:\caD\to G$ with $F\cdot[e_3] = [Fe_3]=f$. In fact the map $f$ is conformal and Lagrangian precisely
when this triple gives such a frame and $|\varphi_z|_\pm=|\varphi_{\bar z}|_\pm$. When we define $e^\psi=|\varphi_z|_\pm$,
the induced metric is given by $2e^{2\psi}|dz|^2$.

This frame has a corresponding Maurer-Cartan $1$-form $\alpha_\pm = F^{-1}dF$, which is a $\fg$-valued $1$-form over $\caD$
(with $\fg$ the Lie algebra of $G$). Let $H=\tfrac12\tr\II_f$ denote the mean curvature of $f$. It can be shown that the
\emph{mean curvature $1$-form}\index{mean curvature 1-form} of $f$, $\sigma_H = \omega(H,df)$, equals $-\tfrac{i}{2}\tr\alpha_\pm$, and therefore $f$
is minimal precisely when $\det F$ is constant. In that case, a straightforward calculation expresses $\alpha_\pm$ in the
form
\begin{equation}\label{eq:alpha}
\alpha_\pm = \begin{pmatrix}
\psi_z  & 0 & e^\psi\\
Qe^{-2\psi} &  -\psi_z & 0\\
0& \mp e^\psi&0
\end{pmatrix}dz\\
+ \begin{pmatrix}
-\psi_{\bz}  & -{\bar Q}e^{-2\psi} & 0\\
0 & \psi_\bz  & e^\psi\\
\mp e^\psi & 0& 0
\end{pmatrix}d\bz,
\end{equation}
where $Q = \g{\varphi_{zz}}{\varphi_{\bar z}}_{\pm}$ so that $Q_3 = Qdz^3$.
The Maurer-Cartan equations for $\alpha_\pm$, $d\alpha_\pm +\tfrac12[\alpha_\pm\wedge\alpha_\pm]=0$, are equivalent to
the elliptic p.d.e.\
\begin{equation}\label{eq:LagToda}
2 \psi_{z\bz} +\lambda e^{2\psi} - |Q|^2e^{-4\psi}=0, \quad \lambda=\pm 1.
\end{equation}
These equations also have a coordinate-invariant form appropriate for the case where we want to consider $\caD$
to be the univesal cover of a compact Riemann surface. Let $\mu$ be a metric of curvature $\kappa_\mu$ on $\caD$,
and write the induced metric as $e^u\mu = 2e^{2\psi}|dz|^2$. Then the previous equation has the form
\begin{equation}\label{eq:GlobalLagToda}
\Delta_\mu u +\lambda 2e^u - 2\|Q_3\|_\mu^2e^{-2u}-2\kappa_\mu=0, \lambda=\pm 1.
\end{equation}

\subsection{The two-dimensional Toda equations for $\fa_2^{(2)}$}

For $\lambda=\pm 1$, the four equations
\begin{eqnarray}
2 \psi_{z\bz} +\lambda e^{2\psi} + |Q|^2e^{-4\psi}=0,\label{eq:affinespheres}\\
2 \psi_{z\bz} +\lambda e^{2\psi} - |Q|^2e^{-4\psi}=0,\label{eq:minlag}
\end{eqnarray}
obtained from, respectively, affine spheres and minimal Lagrangian surfaces,
are real forms of a slight generalization of the two-dimensional
Toda equations for the affine Dynkin diagram $\fa_2^{(2)}$.
There is a two-dimensional Toda equation associated to each affine Dynkin diagram \cite[\S 10]{drinfeld-sokolov84}.
These occur in many types of geometry,
ranging from their role (in the form of the sine-Gordon, or sinh-Gordon, equations) in the theory surfaces of
constant negative Gaussian or constant mean curvature (see, e.g., \cite{bobenko94}), to the their appearance in the study of
superconformal tori in spheres or complex projective spaces \cite{bolton-pw95} and primitive harmonic maps \cite{burstall-pedit94}.
For $\fa_2^{(2)}$ the usual form of the Toda equations is given by
\begin{equation}
\frac{\partial^2}{\partial z\partial w}\log(a^2) - a^2 +a^{-4}=0,
\end{equation}
in which all variables, independent and dependent, are considered to be complex. Notice that there is no explicit
term representing the norm of a cubic differential. That term can be included by using a more general form:
\begin{equation}\label{eq:Toda}
\frac{\partial^2}{\partial z\partial w}\log(a^2) +\lambda a^2 +QRa^{-4}=0,\quad \lambda=\pm 1.
\end{equation}
Here $Q,R$ are independent functions of both complex variables $z$ and $w$.
We have included a sign option $\lambda$, even though this is redundant when all variables are complex, to make the
correspondence with the above real forms simpler.

The principal property of these equations \eqref{eq:Toda} is that they
are the \emph{zero-curvature} (or Maurer-Cartan) equations\index{zero-curvature equations}\index{Maurer-Cartan equations} for the loop of flat $\fa_2\simeq\spl(3,\C)$ connections
\[
d+\alpha_\zeta = d+
\begin{pmatrix} a_za^{-1} & 0 & -\zeta\lambda a\\ \zeta Qa^{-2} &-a_za^{-1}  & 0 \\ 0 &-\zeta \lambda a &0\end{pmatrix}dz
+
\begin{pmatrix} -a_wa^{-1} &  \zeta^{-1}R a^{-2} & 0 \\ 0 & a_wa^{-1} & \zeta^{-1} a \\ \zeta^{-1}a & 0 & 0\end{pmatrix}dw.
\]
Here $\zeta$ is an auxiliary $\C^*$ parameter (often referred to in the integrable systems literature
as the \emph{spectral
parameter}\index{spectral parameter}). We can think of $\alpha_\zeta$ as a $1$-form with values in the loop algebra
\[
L(\fa_2,\nu) = \{X:\C^*\to\fa_2|\ X(-\zeta) = \nu(X(\zeta))\},
\]
where $\nu$ is the outer involution of $\spl(3,\C)$ defined by
\[
\nu(A) = -TA^tT^{-1},\quad A\in\spl(3,\C),\ T = \begin{pmatrix} 0&1&0\\ 1&0&0\\ 0&0&1\end{pmatrix}.
\]
It is this outer involution which corresponds to the symmetry of the Dynkin diagram of $\fa_2$ through which the diagram
for $\fa_2^{(2)}$ arises (see, for example, \cite[Ch X,\ \S 5]{helgason01}).

By imposing different reality conditions on these variables we can obtain all four of the equations above
as follows. Set $w=\bar z$ and $a=e^\psi$, where $\psi$ is a real valued function. Then:
\begin{enumerate}
\item with $R=\bar Q$ and $\lambda=-1$, we obtain the equation for hyperbolic affine spheres;
\item with $R=\bar Q$ and $\lambda=1$, we obtain the equation for elliptic affine spheres;
\item with $R=-\bar Q$ and $\lambda = -1$, we obtain the equation for minimal Lagrangian surfaces in $\CH^2$;
\item with $R=-\bar Q$ and $\lambda = 1$, we obtain the equation for minimal Lagrangian surfaces in $\CP^2$.
\end{enumerate}
These four real forms correspond to the extra requirement that $\alpha_\zeta$ take values in a real form of $L(\fa_2,\nu)$.
Such real forms are obtained by
considering loops which are additionally equivariant with respect to a real involution on $\C^*$ and a real involution on
$\fa_2$. To describe these let us denote, for any $A\in\spl(3,\C)$, its Hermitian transpose by $A^\dagger$ and its
``Lorentz-Hermitian'' transpose by $A^\star$ (i.e., $A\in\su(2,1)$ precisely when $A^\star = -A$).   Given this, to
obtain the four reality conditions on $z,w,Q,R,a$ above, it suffices to require $\alpha_\zeta$ to take values in, respectively:
\begin{enumerate}
\item $\{X\in L(\fa_2,\nu): X(-\bar\zeta^{-1}) = -X(\zeta)^\dagger\}$,
\item $\{X\in L(\fa_2,\nu): X(-\bar\zeta^{-1}) = -X(\zeta)^\star\}$,
\item $\{X\in L(\fa_2,\nu): X(\bar\zeta^{-1}) = -X(\zeta)^\star\}$,
\item $\{X\in L(\fa_2,\nu): X(\bar\zeta^{-1}) = -X(\zeta)^\dagger\}$.
\end{enumerate}
In cases (c) and (d) the real involution on $\C^*$ has a fixed curve (the unit circle) so that when $\alpha_\zeta$ is evaluated
on this curve it takes values in the respective real form of $\fa_2$. Up to a constant gauge transformation,
we obtain the Maurer-Cartan forms in \eqref{eq:alpha}. But in cases (a) and (b) the real involution on $\C^*$ has no fixed
points. Hence the real forms $\su(3)$ and $\su(2,1)$ do not play an explicit role in the corresponding geometry.

For harmonic maps this loop algebra observation plays an important role in understanding the construction of certain classes
of solutions (see, e.g.\ \cite{burstall-pedit95,dorfmeister-pedit-wu98,mcintosh98,mcintosh03,segal89,uhlenbeck89}). For global solutions this has only been successful when
the domain is either $S^2$ or a torus and the codomain is $\CP^2$.

\subsubsection{The harmonic map equations from the loop of flat connections.} The four reality conditions above give us four
slightly different loops of flat connections. Using the parameters $\e=\pm 1$ and $\lambda = \pm 1$ we can write the local
connection $1$-forms as
\begin{equation}\label{eq:realalpha}
\alpha_\zeta=
\begin{pmatrix}
\psi_z  & 0 & -\zeta\lambda e^\psi\\ \zeta Qe^{-2\psi} &  -\psi_z & 0\\ 0&  -\zeta\lambda e^\psi&0
\end{pmatrix}dz
+ \begin{pmatrix}
-\psi_{\bz}  & \zeta^{-1}\e {\bar Q}e^{-2\psi} & 0\\ 0 & \psi_\bz  & \zeta^{-1}e^\psi\\ \zeta^{-1}e^\psi & 0& 0
\end{pmatrix}d\bz.
\end{equation}
The fact that this connection is flat \emph{for all} $\zeta\in\C^*$ tells us that in each of the four cases there exists a
local solution $F$ to the equation $F^{-1}dF = \alpha_1$, taking values in $\mathrm{ SL}(3,\C)$, which frames a harmonic map into a
symmetric space of one of the real forms $\mathrm{SU}(3)$, $\mathrm{SU}(2,1)$ or $\mathrm{SL}(3,\R)$. The argument is slightly different between the
two cases $\e=\pm 1$, but in principle it is based on the following well-known result.
\begin{thm}[cf.\ \cite{burstall-pedit94}]
Let $G/K$ be a symmetric space of a real reductive Lie group $G$, and write the symmetric-space decomposition of the
Lie algebra of $G$ as $\fg=\fk+\fm$, where $\fk$ is the Lie algebra of $K$. Then a smooth map $f:D\to G/K$ of an open
domain $D\subset\C$ is a harmonic map if and only if it admits a frame $F:D\to G$ whose Maurer-Cartan $1$-form
$\alpha = F^{-1}dF$ satisfies the equations
\[
d*\alpha_\fm +[\alpha_\fk\wedge *\alpha_\fm]=0,
\]
where $\alpha = \alpha_\fk+\alpha_\fm$ is the symmetric space splitting of $\alpha$.

Conversely, suppose $\alpha\in\Omega_1(D)\otimes\fg$ is a Lie algebra valued $1$-form over $D$ for which
\begin{equation}\label{eq:flatloop}
d\alpha_\zeta + \tfrac{1}{2}[\alpha_\zeta\wedge\alpha_\zeta]=0,
\end{equation}
where
\begin{equation}\label{eq:loop}
\alpha_\zeta = \zeta\alpha_\fm^{1,0} + \alpha_\fk + \zeta^{-1}\alpha_\fm^{0,1}.
\end{equation}
By integrating the
Maurer-Cartan equations for $\alpha$ we obtain a frame $F:D\to G$ for a harmonic map $f:D\to G/K$.
\end{thm}
The point is that the condition \eqref{eq:flatloop} is exactly the condition that $\alpha$
simultaneously satisfies the Maurer-Cartan equations and the harmonic map equations.

Now we can explain how the loop of flat connections in \eqref{eq:realalpha} is related to harmonic maps.
\begin{enumerate}
\item For $\e = 1$, recall that $\alpha$ arose from the $\mathrm{SL}(3,\C)$ frame for an affine sphere
\begin{equation}\label{eq:sphereframe}
F=(\frac{1}{\sqrt{2}e^{\psi}}f_z\ \frac{1}{\sqrt{2}e^{\psi}}f_\bz\ \xi) =
(\frac{1}{\sqrt{2}e^{\psi}}f_x\ \frac{1}{\sqrt{2}e^{\psi}}f_y\ \xi)
\begin{pmatrix} 1/2 & 1/2 & 0\\ -i/2 & i/2 & 0 \\ 0 & 0 & 1\end{pmatrix}.
\end{equation}
This shows that $F$ is related to an $\mathrm{SL}(3,\R)$ frame by a constant gauge, and therefore $\alpha$ takes values in a real
subalgebra $\fg\subset \spl(3,\C)$ which is conjugate to $\spl(3,\R)$. It is easy to check that the outer involution $\nu$
preserves $\fg$, on which it induces a symmetric space-decomposition $\fg = \fk+\fm$, with $\fk\simeq\so(3,\R)$. Further,
$\alpha_\zeta$ in \eqref{eq:realalpha} has the form \eqref{eq:loop} for this splitting. The Toda equations are the condition
that this satisfies \eqref{eq:flatloop}, and therefore we obtain from $F$ (after a constant gauge transformation) a
harmonic map into $\mathrm{ SL}(3,\R)/\mathrm{SO}(3,\R)$.
\item For $\e=-1$ we use a slight generalization of the previous theorem. The loop of flat connections $\alpha_\zeta$ is
equivariant not only for the outer involution $\nu$ but for an order-$6$ automorphism of the form $\nu\circ\sigma$, where
$\sigma$ is the Coxeter automorphism of $\spl(3,\C)$ ($\sigma$ is an inner automorphism of order $3$ which commutes with $\nu$,
cf.\ \cite{mcintosh03}). In this case the equations \eqref{eq:flatloop} tell us (cf.\
\cite{burstall-pedit94}, \cite{mcintosh03}, \cite{loftin-mcintosh13}) that $F$ frames a
primitive harmonic map into a $6$-symmetric space $G/S$, where $G=\mathrm{SU}(3)$ for $\lambda = 1$ and $G=\mathrm{ SU}(2,1)$ for $\lambda=-1$.
Here $S$ is an $S^1$ subgroup of the maximal torus of diagonal matrices in $G$ (those fixed by the automorphism). Thus we have
a homogeneous projection $G/S\to G/K$, where $G/K$ is $\CP^2$ or $\CH^2$ according to the sign of $\lambda$. Under this
homogeneous projection a primitive harmonic map is still harmonic, therefore $F$ frames a harmonic map into either $\CP^2$ or
$\CH^2$.
\end{enumerate}

\subsection{Holomorphic representations}
As we have seen above, solutions to the equations
$$ 2\psi_{z\bar z} \pm |Q|^2e^{-4\psi} \pm e^{2\psi} =0 $$
are the Toda lattice for $\mathfrak a_2^{(2)}$.  The remaining cases
$$2\psi_{z\bar z}\pm |Q|^2e^{-4\psi} = 0$$
do not come from the Toda lattice, but are
completely integrable in that they can easily be derived from holomorphic data without solving any additional PDEs.

\subsubsection{Weierstrass representation for parabolic affine spheres}\index{Weierstrass representation}

A parabolic affine sphere in $\re^{n+1}$ with affine normal $e_{n+1}$ is given by the graph of a convex function $u$ satisfying the Monge-Amp\`ere equation $\det u_{ij} =1$.  In dimension two, there is a classical relationship between solutions to the Monge-Amp\`ere equation and harmonic functions (this was know to Darboux). This naturally leads to a description of parabolic affine spheres in $\re^3$ by holomorphic functions.

In fact, there is a natural Weierstrass-type formula for parabolic affine spheres, which (in the more general case of affine maximal surfaces) is originally due to Terng \cite{terng83}.  See also \cite{calabi88,li89,greene-svy90}.  We give the version of Ferrer-Mart\'inez-M\'ilan \cite{ferrer-mm99}: On a simply connected domain $\mathcal D\subset\C$, let $F,G$ be two holomorphic functions satisfying $|F'|<|G'|$. Then
$$ \left(\frac12(G+\bar F), \frac13 (|G|^2-|F|^2), \frac14\, {\rm Re}(FG) - \frac12 \int F\,dG\right)$$ is a parabolic affine sphere with affine normal $(0,0,1)$.  All such parabolic affine spheres can be described in this way.

\subsubsection{Minimal Lagrangian surfaces in $\C^2$.} \label{c2-subsec}
All minimal surfaces in Euclidean $\re^n$ have a Weierstrass representation, since for $f\!:\Sigma\to \re^n$ the minimal surface equations are $\Delta f=0$.  Then one can replace $3$ by $n$ in the classical Weierstrass-Enneper argument to retrieve the holomorphic data.  However, we emphasize a different construction relating minimal Lagrangian surfaces in $\C^2$ to special Lagrangian surfaces and then to complex curves in a rotated complex structure.

In $\C^2$ (or any Calabi-Yau manifold), minimal Lagrangian submanifolds are closely related to special Lagrangian submanifolds\index{special Lagrangian}.  A surface in $\co^2$ is called special Lagrangian if it is calibrated by ${\rm Re}(dz^1\wedge dz^2)$.  Harvey-Lawson  \cite{harvey-lawson82} show that any minimal Lagrangian surface in $\co^2$ is calibrated by ${\rm Re}(e^{i\theta}dz^1\wedge dz^2)$, where $\theta$ is a real constant. In other words, the minimal Lagrangian surface is special Lagrangian up to a constant Lagrangian angle (or phase) $\theta$.

In dimension four, a Calabi-Yau manifold is automatically hyper-K\"ahler (this follows since a Calabi-Yau manifold has local holonomy contained in $\mathrm{SU}(2)\sim \mathrm{Sp}(1)$, and 4-manifolds with holonomy in $\mathrm {Sp}(1)$ are hyper-K\"ahler).  A hyper-K\"ahler manifold is a manifold with a Riemannian metric and three complex structure tensors $I,J,K$ which satisfy the quaternionic relations $I^2=J^2=K^2=IJK= -{\rm id}$, and the metric is K\"ahler with respect to each complex structure.  Moreover, for $\alpha,\beta,\gamma$ real constants so that $\alpha^2+\beta^2+\gamma^2=1$, $\alpha I +\beta J + \gamma K$ is also a complex structure.  A special Lagrangian submanifold of a hyper-K\"ahler 4-manifold can be characterized as a complex submanifold with respect to a rotated complex structure.  This is because there is a hyper-K\"ahler rotation which takes ${\rm Re}(dz^1\wedge dz^2)$ to the K\"ahler form $\omega$, and $\omega$ calibrates complex surfaces.

In particular, a minimal Lagrangian surface in $\C^2$ is special Lagrangian up to a rotation of the holomorphic volume form $dz^1\wedge dz^2$ by a constant Lagrangian angle.

\begin{prop}
Let $I$ be the standard complex structure on $\C^2$, and consider the standard flat metric as a hyper-K\"ahler metric.
Any minimal Lagrangian surface in $\C^2$ is holomorphic with respect to one of the complex structures  orthogonal to $I$.
\end{prop}

\section{Global theory and Applications.}

Now we address solving the integrability conditions on a Riemann surface $\Sigma$, instead of just on a simply connected domain $\mathcal D\subset \co$.  So let $\Sigma$ be a Riemann surface with a conformal background metric $\mu$ and a holomorphic cubic differential $Q$.  Here in local coordinates the affine metric is given by $2e^{2\psi}|dz|^2=e^u\mu$. Then the local equations we discuss above are of the form
\begin{equation} \label{generic-eq}
\Delta u + 2\epsilon\|Q\|^2 e^{-2u} + 2\lambda e^u -2\kappa = 0,
\end{equation}
where $\Delta$ is the Laplace operator with respect to $\mu$, $\|\cdot \|$ is the induced norm on cubic differentials, and $\kappa$ is the Gauss curvature.  Here $\lambda\in\{-1,0,1\}$ and $\epsilon=\pm1$.

The global properties of (\ref{generic-eq}) depend heavily on the signs $\epsilon$ and $\lambda$, as the maximum principle underlies the analysis in many cases.  As an illustration, we provide a proof of uniqueness in the best case for these signs---that of hyperbolic affine spheres ($\epsilon=1$ and $\lambda=-1$).
\begin{prop}
Let $\Sigma$ be a closed Riemann surface with a hyperbolic ($\kappa=-1$) background metric $\mu$, and let $Q$ be a holomorphic cubic differential on $\Sigma$.  Then there is at most one $C^2$ solution to $$\Delta u + 2\|Q\|^2e^{-2u} -2e^u +2=0.$$
\end{prop}
\begin{proof}
Let $u,v$ be two $C^2$ solutions.  Then at the maximum point $p$ of $u-v$, we have
\begin{eqnarray*}
\Delta(u-v)(p) &\le &0,\\
\Delta u(p) &\le & \Delta v(p), \\
-2\|Q(p)\|^2 e^{-2u(p)} + 2e^{u(p)} - 2 &\le & -2\|Q(p)\|^2e^{-2v(p)} + 2e^{v(p)} - 2, \\
u(p)&\le& v(p), \\
(u-v)(p)&\le&0.
\end{eqnarray*}
(The key step is to recognize that $2\|Q\|^2e^{-2u}-2e^u-2$ is a decreasing function of $u$, and so the fourth line follows from the third.)  Since $p$ is the maximizer of $u-v$, we see that $u-v\le0$ on all of $\Sigma$. By switching the roles of $u$ and $v$, we see that $v-u\le0$ on $\Sigma$, and so $u=v$ identically on $\Sigma$.
\end{proof}

The global theories for other values of $\lambda$ and $\epsilon$ are less well behaved, and a successful application of the elliptic theory can depend on balancing the good terms versus the bad terms.  An exception is the seemingly worst case of minimal Lagrangian surfaces in $\CP^2$: $\lambda=1$ and $\epsilon=-1$.  In this case, when $\Sigma$ has abelian fundamental group and $Q$ is constant, powerful techniques of integrable systems can be brought to bear.

\subsection{Hyperbolic affine spheres} \label{has-subsec}

\subsubsection{Wang's equation}

We begin with a Riemann surface $\Sigma$ equipped with a holomorphic cubic differential $Q$ and a background conformal Riemannian metric $\mu$. The local integrability condition (\ref{aff-sph-eq-local}) becomes
\begin{equation} \label{global-has-eq}
\Delta u +2e^{-2u}\|Q\|^2 - 2e^u -2\kappa = 0,
\end{equation}
where $\Delta$ is the Laplacian with respect to $\mu$, $\|\cdot\|$ is the induced norm on cubic differentials and $\kappa$ is the Gauss curvature of $\mu$.

Before we talk about the main case of interest, we take care of one trivial case.  If $\Sigma$ is an elliptic curve with flat coordinate $z$, then each holomorphic cubic differential is of the form $Q=c\,dz^3$, where $c\in\C$ is a constant.  Also equip $\Sigma$ with the flat metric $|dz|^2$ so that $\kappa=0$ and (\ref{global-has-eq}) becomes
$$\Delta u + 2e^{-2u}|c|^2 - 2e^u = 0.$$
In this case, integration shows there is no solution for $c=0$, but there is an explicit constant solution $u = \frac13\log(8|c|^2)$ for nonzero $c$.  In this case, the structure equations can be integrated explicitly to find that the hyperbolic affine sphere is just the \c{T}i\c{t}eica-Calabi example.  See e.g.\ \cite{loftin02c}.

This problem was first formulated by  C.P.\ Wang, who studied the solutions to (\ref{global-has-eq})  on a closed hyperbolic Riemann surface, in order to produce more examples of hyperbolic affine spheres in $\re^3$.  Wang's existence proof for (\ref{global-has-eq}) is flawed, however.  The first existence proof is due to Labourie, who arrived at the same differential-geometric structures as the hyperbolic affine spheres from a different, more intrinsic point of view \cite{labourie97,labourie07}.

If $\Sigma$ is a compact Riemann surface and $\mu$ is a hyperbolic metric (i.e., $\kappa=-1$), then $\Sigma$ must have genus at least 2, and the space of holomorphic cubic differentials over $\Sigma$ has complex dimension $5g-5$. Equation (\ref{global-has-eq}) becomes
$$ \mathcal L(u)\equiv\Delta u + 2e^{-2u}\|Q\|^2 -2e^u +2 = 0.$$

\begin{prop} \label{solve-global-has-eq}
Let $\Sigma$ be a closed Riemann surface equipped with a hyperbolic metric $\mu$ and a holomorphic cubic differential $Q$. Then there is a unique solution $u$ to (\ref{global-has-eq}). The solution $u$ is $C^\infty$.
\end{prop}

\begin{proof}
This is a semilinear elliptic equation, and there are many straightforward techniques to show existence. We mention the method of sub and supersolutions (see e.g. Schoen-Yau \cite{schoen-yau94}).  We check $\mathcal L(0)\ge 0$  and $\mathcal L(\log m)\le0$, where $m$ is the positive root of
$$ x^3-x^2-\max_{\Sigma}\|Q\|^2  = 0.$$
Since $0\le \log m$ as well, the theory of sub and supersolutions guarantees the existence of a smooth solution $u$ between $0$ and $\log m$. We have already seen this solution to be unique.
\end{proof}

\subsubsection{Properly embedded hyperbolic affine spheres}

It is also relevant to recount the theory of properly embedded hyperbolic affine spheres in $\re^{n+1}$, which is
due to Cheng-Yau \cite{cheng-yau77,cheng-yau86} and Calabi-Nirenberg \cite{calabi-nirenberg74}.
\begin{thm} \label{classify-has}
Let $\mathcal C\subset \re^{n+1}$ be an open convex cone which is nondegenerate in that it contains no line. Let $v$ be the vertex of $\mathcal C$.  Then there is a hyperbolic affine sphere $L$ whose center is $v$ and which is asymptotic to the boundary of $\mathcal C$.  $L$ is unique up to homotheties centered at $v$, and is unique if we require the affine shape operator to be minus the identity.  The Blaschke metric on $L$ is complete.

Moreover, if $L$ is an immersed hyperbolic affine sphere in $\re^{n+1}$ with complete Blaschke metric, then $L$ is properly embedded, and there is a nondegenerate cone $\mathcal C$ with vertex the center of $L$ so that $L$ is asymptotic to its boundary.
\end{thm}

This theorem was conjectured by Calabi in \cite{calabi72}. We explain the relationship between $\mathcal C$ and $L$ a bit further here. The picture is that of the standard example of an elliptic hyperboloid, which is asymptotic to a cone over an ellipsoid.  In general, given $v$ and $L$, the open cone $\mathcal C$ can be defined as the interior of the convex hull of $L$ and $v$.

We also note the following auxiliary proposition:

\begin{prop}
By a rigid motion of $\re^{n+1}$, normalize the affine spheres so that the center $v=0$ and $S=-I$. Under the conormal map, the image of $L$ is the hyperbolic affine sphere asymptotic to the cone over the dual cone $\mathcal C^*$.
\end{prop}

The history and ingredients of the proof of this theorem are somewhat complicated, with the main results of Cheng-Yau clarified later by Gigena \cite{gigena81}, Li \cite{li90,li92}, and Sasaki \cite{sasaki80}.  We refer the reader to \cite{loftin10} for a full account.  Uniqueness is a straightforward application of the maximum principle to the elliptic Monge-Amp\`ere equation $\det u_{ij} = (-\frac1u)^{n+2}$ discussed in Subsection \ref{affine-sph-subsec} above.

\subsubsection{Convex $\rp^2$ structures}

Given a Lie group $G$ and a homogeneous space $X$ on which $G$ acts transitively, an $(X,G)$ structure on a manifold $M$ is given by a maximal atlas on $M$ of coordinate charts in $X$ whose transition functions are locally constant elements of $G$\index{$(X,G)$ structure}. So each point in $M$ has a coordinate neighborhood which looks like $X$, and these neighborhoods are glued together by elements of $G$, which we regard as automorphisms of $X$.  When $X = \rp^n$ and $G = \pgl{n+1}$, the structure is called an $\rp^n$ structure, or a real projective structure\index{real projective structure}\index{$\rp^n$ structure}.

An $\rp^n$ structure on $M$ is called \emph{convex} if is induced from a quotient of a convex domain $\Omega$ in an affine $\re^n\subset \rp^n$ by a subgroup of $\pgl{n+1}$ acting discretely and properly discontinuously\index{convex $\rp^n$ structure}\index{$\rp^n$ structure!convex}.  The $\rp^n$ structure is \emph{properly convex} if it is convex and the closure $\bar\Omega$ is disjoint from a hyperplane in $\rp^n$\index{properly convex $\rp^n$ structure}\index{$\rp^n$ structure!properly convex}.  It is a result of Kuiper that on every closed oriented surface of genus at least 2, any convex $\rp^2$ structure is properly convex \cite{kuiper54}.

The proper convexity of $\Omega\subset\rp^2$ means that the cone $\mathcal C$ over $\Omega$ does not contain any lines.  In particular, there is a unique hyperbolic affine sphere $L$ asymptotic to the boundary of $\mathcal C$ with center at the origin and shape operator minus the identity.  Under the projection from $\mathcal C \to \Omega$, $L$ is mapped diffeomorphically onto $\Omega$.  Moreover, oriented projective-linear actions on $\Omega$ lift to unimodular linear actions on $\mathcal C$. By the uniqueness of $L$, these linear actions must act on $L\subset\mathcal C$ as well. The affine invariants of $L$ are thus projective invariants on $\Omega$ which descend to the quotient surface.

The following theorem was proved independently by Labourie \cite{labourie97,labourie07} and the first author \cite{loftin01}.
\begin{thm} \label{rp2-cubic}
Let $S$ be a closed oriented surface of genus $g\ge 2$. The convex $\rp^2$ structures on $S$ are in one-to-one correspondence with pairs $(\Sigma,Q)$ where $\Sigma$ is a conformal structure on $S$ and $Q$ is a holomorphic cubic differential.
\end{thm}

\begin{proof}
This follows from  Propositions  \ref{solve-global-has-eq} and \ref{complete-has-rp2}.
\end{proof}

This offers a new proof of a theorem of Goldman

\begin{cor}\cite{goldman90a}
The deformation space of marked convex $\rp^2$ structures on a closed oriented surface $S$ of genus $g\ge2$ is homeomorphic to $\re^{16g-16}$.
\end{cor}

\begin{proof}
The theorem describes the deformation space as the total space of a vector bundle over Teichm\"uller space with fibers $H^0(\Sigma,K^3)$ the space of cubic differentials.  Teichm\"uller space has complex dimension $3g-3$ and is diffeomorphic to $\re^{6g-6}$. Each fiber has complex dimension $5g-5$.
\end{proof}

Theorem \ref{rp2-cubic} puts a natural complex structure on the deformation space $\mathcal G_S$ of convex $\rp^2$ structures which is invariant under the mapping class group.  There is a natural symplectic structure on this space due to Goldman \cite{goldman90b}.  It is not clear whether these two structures fit together to form a K\"ahler structure.  There is another way to attempt to construct a K\"ahler structure. Using a  form of Kodaira-Spencer correspondence, one can identify the tangent space to the deformation space with $H^1(\Sigma,\mathcal F)$, where $\mathcal F$ is a flat $\Sl3$ principal bundle induced by the $\rp^2$ structure.  The hyperbolic affine sphere induces a natural metric on this bundle which induces a metric of Weil-Petersson type on the deformation space.  This metric and the symplectic form produce an almost K\"ahler structure on $\mathcal G_S$, but it is not clear if the almost-complex structure is integrable. This theory is due to Darvishzadeh-Goldman \cite{d-goldman}, although they use an invariant hypersurface different from the hyperbolic affine sphere. Recently Q.\ Li has studied the metric on $\mathcal G_S$ derived from the hyperbolic affine sphere and has found that Teichm\"uller space sits in $\mathcal G_S$  totally geodesically \cite{qli13}.

\begin{prop} \label{complete-has-rp2}
Let $S$ be any oriented surface.  Then the properly convex $\rp^2$ structures on $S$ are in one-to-one correspondence with triples $(\Sigma,Q,e^u\mu)$, where $\Sigma$ is a conformal structure on $S$, $Q$ is a holomorphic cubic differential, and $e^u\mu$ is a complete conformal Blaschke metric corresponding to $\Sigma$ and $Q$ by (\ref{global-has-eq}).
\end{prop}

\begin{proof}
Assume $S$ admits a properly convex $\rp^2$ structure.  Then we see $S = \Omega/\Gamma$, where $\Omega \subset\rp^2$ is properly convex and $\Gamma$ is a discrete group of projective transformations acting faithfully and properly discontinuously.  Consider the cone $\mathcal C\subset\re^3$ over $\Omega$. Then Theorem \ref{classify-has} gives a unique (properly normalized) hyperbolic affine sphere $K$ asymptotic to the boundary of $\mathcal C$. $K$ has complete Blaschke metric and $K$ is diffeomorphic to $\Omega$ under projection.  We may canonically lift the action of $\Gamma$ to $\Sl3$. Then the uniqueness of $K$ shows that it is acted upon by $\Gamma$.  All the affine invariants on $K$, such as the Blaschke metric and the cubic form, then pass to the quotient.  The Blaschke metric induces a conformal structure on $S$ canonically determined by the $\rp^2$ structure, and the cubic form becomes a holomorphic cubic differential as in Subsection \ref{affine-sph-subsec} above.

To prove the converse, assume $S$ admits a conformal structure $\Sigma$, a holomorphic cubic differential $Q$, and a complete conformal metric $e^u\mu$ satisfying (\ref{global-has-eq}). Choose a basepoint $p\in\Sigma$ and a complexified frame $F(p)\equiv\{f_z(p),f_{\bar z}(p),f(p)\}$ in $\re^3$ as in Theorem \ref{aff-sph-integrate} above.  Let $\tilde\Sigma$ be the universal cover of $\Sigma$ with basepoint $p$. Then Theorem \ref{aff-sph-integrate} shows that there is a unique hyperbolic affine sphere immersion $f$ of $\tilde \Sigma$ into $\re^3$ which satisfies the initial conditions at $p$ whose cubic form is $Q$ and whose Blaschke metric is $e^u\mu$.  The Blaschke metric on $\tilde\Sigma$ is complete by assumption, and so Theorem \ref{classify-has} shows that $K=f(\tilde\Sigma)$ is the unique hyperbolic affine sphere asymptotic to the boundary of a convex cone $\mathcal C\subset \re^3$.  For a local coordinate $z$, we may define a connection $D=d+\alpha$ on $E=L \oplus T\Sigma $ by (\ref{conn-form-has}) below, where $f_z,f_{\bar z}$ form a complexified frame on $T\Sigma$.  Equation (\ref{global-has-eq}) shows $D$ is flat, and for any loop $\beta\in\pi_1\Sigma$ based at $p$, we define the holonomy ${\rm hol}(\beta)$ to be the inverse of the parallel transport of $D$ along $\beta$.  Then the discussion in \ref{hol-dev-subsection} below implies that ${\rm hol}(\beta)$ acts on $f$ in the sense that if $y\in\tilde\Sigma$ and $\beta y$ is the image of $y$ under the deck transformation, then $f(\beta y) = {\rm hol}(\beta)f(y)$.  Thus the hyperbolic affine sphere $K=f(\tilde D)$ is acted upon by the holonomy representation ${\rm hol}(\pi_1\Sigma)$.  Now Theorem \ref{classify-has} shows ${\rm hol}(\pi_1\Sigma)$ acts on $\mathcal C$ as well.  By projecting $\mathcal C$ to a properly convex domain in $\rp^2$, we induce a convex $\rp^2$ structure on $S$.
\end{proof}

\subsubsection{Holonomy and developing map} \label{hol-dev-subsection}

For $M$ a connected $(X,G)$-manifold, choose $p\in M$ and an $X$-coordinate ball on a neighborhood of $p$.  These choices induce the \emph{development-holonomy pair}\index{development-holonomy pair}.

Let $\pi_1M$ and $\tilde M$ be respectively the fundamental group and universal cover of $M$ with basepoint $p$.
Then the developing map ${\rm dev}\!:\tilde M\to X$ is given by analytically continuing the coordinate chart around $p$ along any loop.  The holonomy map ${\rm hol}\!: \pi_1M \to G$ satisfies
\begin{equation}\label{dev-hol-diag}\begin{CD}
\tilde M @>{\rm dev}>> \RP^2 \\
@V{\gamma}VV  @VV{{\rm hol}(\gamma)}V \\
\tilde M @>{\rm dev}>> \RP^2
\end{CD}\end{equation}
Conversely, a pair of ${\rm dev}\!:\tilde M \to X$  a local diffeomorphism and ${\rm hol}\!:\pi_1M\to G$  a homomorphism so that (\ref{dev-hol-diag}) is satisfied determine an $(X,G)$ structure on $M$.

In the case of $\RP^2$ structures ($X=\RP^2$ and $G=\pgl3$), we relate the $\RP^2$ structure to the structure equations for hyperbolic affine spheres developed above.  We follow the treatment of Goldman \cite{goldman90b}.  Let $M$ be a surface.  A connection $\na$ on $TM$ is called \emph{projectively flat}\index{projectively flat connection} if its geodesics (as sets) locally coincide with geodesics of a flat torsion-free connection. If $\na$ is projectively flat and torsion-free, then it can be uniquely lifted to a flat connection (the \emph{normal projective connection}\index{normal projective connection}) $D$ on $E=TM\oplus L$, where $L$ is the trivial line bundle with a distinguished nowhere-vanishing section $f$, and for $X,Y$ tangent vector fields,
\begin{eqnarray*}
D_XY &=& \na_XY + h(X,Y)f, \\
D_X f &=& X.
\end{eqnarray*}
 Here $h$ is a symmetric tensor.  This is exactly the form of the structure equations for the hyperbolic affine sphere.

The flat connection $D$, together with a nowhere vanishing section $f$ of $L$ can be used to define the dev-hol pair as follows.  Choose a basepoint $x\in M$, and consider the universal cover $\tilde M$ and fundamental group $\pi_1M$ with respect to this basepoint.  Then let $\gamma\in\pi_1M$ be represented by a loop with basepoint $x$, and let $\beta$ be a path from $x$ to $y$ in $M$.  Lift $\beta$ to a path on $\tilde M$ going from $\tilde x$ to $\tilde y$.  Let $E_x$ be the fiber of $E$ over $x$, and let $P$ be the projection map from $E_x-\{0\}$ to the projective space $P(E_x)\sim\RP^2$.  Let $\Pi_\beta\!:E_x\to E_y$ denote the parallel transport for $D$ along $\beta$.  Since $D$ is flat, $\Pi_\beta$ is independent of $\beta$ in a homotopy class.  Then if we define
$$ {\rm dev}(\tilde y) = P(\Pi_{\beta^{-1}}(f(y))), \qquad {\rm hol}(\gamma) = \Pi_\gamma^{-1},\index{holonomy}\index{developing map}$$
 (\ref{dev-hol-diag}) is satisfied.

Now we show how to relate this to the structure equations above. Recall that for the frame $F = (f_z\,\,f_{\bar z}\,\,f)^\top$, $f$ is the immersion of the hyperbolic affine sphere.  We project $f$ from $\re^3$ to $\rp^2$ via $f\mapsto P(f)$. This is the developing map coming from the structure equations
$$ {\rm dev}'(\tilde y) = P(f(\tilde y))$$
for $\tilde y\in \tilde\Sigma$, and $f$ evolved along a path from $\tilde x$ to $\tilde y$.
\begin{prop}
There is a projective map $\chi$ from $P(\re^3)$ to $P(E_x)$ so that ${\rm dev}(\tilde y) = \chi\circ {\rm dev}'(\tilde y)$.
\end{prop}

We prove this proposition in the next few paragraphs.

With respect to the frame $F$, consider a section $s = SF$, where $S = (s_1\,\,s_2\,\,s_3)$. For the connection matrix $\alpha$, we have $$\na s = (dS)F + S(\alpha F),$$ and so along the path $\beta$, the parallel transport equation is
$$ 0= \na_{\dot\beta} s_i = \frac d{dt}\,s_i(\beta) + s_j(\beta)\langle \alpha^j_i,\dot\beta\rangle,$$
where $\langle\cdot,\cdot\rangle$ is the pairing between one-forms and tangent vectors.
If $\beta\!:[0,1]\to M$ with $\beta(0)=x$ and $\beta(1)=y$, and if we are in a single coordinate chart in $\tilde M$, the parallel transport of $D$ along $\beta$ can be represented in an appropriate frame as $S(1)$, where $S$ is the solution to
\begin{equation} \label{parallel-transport}
 S(0) = I, \qquad \frac d{dt}\, S = - S\,\langle \alpha,\dot\beta\rangle.
 \end{equation}

\begin{lem}
Consider a frame $e_1,e_2,e_3$ of $E$ along $\beta$, and write the connection form of $D$ as $\alpha$ in this frame. Let $J(t) = \Pi_{\beta|_{[0,t]}^{-1}}$. Then $J(0)=I$, the identity transformation, and $\dot J(t) = \langle \alpha(t),\dot\beta(t)\rangle J(t)$.
\end{lem}
\begin{proof}
For small $h$,
$$ J(t+h) = \Pi_{\beta|_{[0,t]}^{-1}} = \Pi_{\beta|_{[0,t]}}^{-1} = (\Pi_{\beta|_{[0,t]}} \Pi_{\beta|_{[t,t+h]}})^{-1}  = \Pi_{\beta|_{[t,t+h]}^{-1}} J(t),$$
which implies $\dot J(t) = \frac d{dh}|_{h=0} \Pi_{\beta|_{[t,t+h]}^{-1}} J(t)$. Now the tangent vector along the path $\beta|_{[t,t+h]}^{-1}$ is $-\dot\beta$, and thus by (\ref{parallel-transport}) the parallel transport solves the equation $\dot S = \langle \alpha,\dot \beta\rangle S$ and $S(0)=I$.
\end{proof}

Now consider the specific situation above.  For the complexified frame $F = (f_z\,\,f_{\bar z}\,\,f)^\top$, $F$ evolves along a path $\beta$ by $$ \dot F = \langle \alpha,\dot\beta\rangle F,$$ where as above, the connection form
\begin{equation}
\label{conn-form-has}
\alpha = \left( \begin{array}{ccc} 2\psi_z &Qe^{-2\psi} & 0 \\ 0&0& e^{2\psi} \\ 1&0&0
\end{array}\right) dz + \left(\begin{array}{ccc} 0&0& e^{2\psi} \\ \bar Qe^{-2\psi} &2\psi_{\bar z} &0 \\ 0&1&0 \end{array} \right)d\bar z.
\end{equation}
If in addition, we set the initial condition to $F(0)=F_0$, the previous lemma implies that $F(t)=J(t)F_0$ for all $t$.

To show the two notions of developing map are the same up to projective equivalence, recall their definitions.   For the evolution of the frame $F$, the developing map ${\rm dev}$ is given by $P(f)$ for $f=e_1\in\re^3$ the component of $F$.  On the other hand, ${\rm dev}'(\tilde y) = P(J(1)(\rho(y))$ where $\rho$ is a nonvanishing section of $L\subset E$.  Recall that $E = L \oplus TM$ and $L$ is a trivial bundle.  $TM$ is spanned by the tangent vectors $f_z,f_{\bar z}$, while $L$ is spanned by $f$.  In particular, we identify $\rho(y)$ as a nonzero multiple of $f(y)$.  So ${\rm dev}'(\tilde y)$ is given by the $f$ component of $F(1)$; this amounts to specifying the first column of $F(1)$.  Since $\rho(y)$ is equivalent to $f(y)$ up to a scalar, the action of $J(1)(\rho(y))$ involves specifying the same component of $J(1)$. This shows that ${\rm dev} (\tilde y) = {\rm dev}'(\tilde y)$.

(One can also extend a projectively flat connection on $TM$ to a connection on a different bundle of rank $n+1$, the tractor bundle.  See for example \cite{baraglia-thesis} or \cite{fox13}. The normal projective connection mentioned here seems better suited to hyperbolic affine spheres.)

We should also mention that Labourie has developed this theory from a more intrinsic point of view \cite{labourie97,labourie07}.  Although Labourie's point of view is quite different, the underlying construction is the same.

\subsubsection{Noncompact and limiting cases}
On a marked closed oriented surface $S$ of genus $g$ at least 2, the identification of convex $\rp^2$ structures with pairs $(\Sigma,Q)$ still holds if one forgets the marking.  In other words, one can take the quotient of the deformation space\index{deformation space} of marked convex $\rp^2$ structures by the mapping class group action (we call this quotient the \emph{moduli space of convex $\rp^2$ structures} on $S$\index{moduli space of convex $\rp^2$ structures}).  The quotient of Teichm\"uller space $\mathcal T_g$ by the mapping class group, the moduli space of Riemann surfaces, is well known to be a complex quasi-projective orbifold $\mathcal M_g$.  The bundle of cubic differentials is an orbifold vector bundle over $\mathcal M_g$.  The Deligne-Mumford compactification $\overline{\mathcal M_g}$ is a compact orbifold formed by attaching singular Riemann surfaces with nodes locally of the form $\{zw=0\}\subset \co^2$. Each such nodal Riemann surface at the boundary of moduli space can be approximated by smooth Riemann surfaces with thin necks pinched by letting $\epsilon\to0$ for each neck of the form $\{zw=\epsilon\}\subset\co^2$.  The bundle of cubic differentials over $\mathcal M_g$ can be extended to the bundle of \emph{regular cubic differentials}\index{regular cubic differential}\index{cubic differential!regular} over $\overline{\mathcal M_g}$.  A regular cubic differential $Q$ on a nodal surface is allowed to have poles of order at most 3 at each node along each sheet $\{z=0\}$ or $\{w=0\}$.  Each such pole has a \emph{residue} $R$ so that $Q = R\,z^{-3}\,dz^3 + O(z^{-2})$.  For a regular cubic differential, the residues of the two sheets of any node must sum up to 0.

From the point of view of the Deligne-Mumford compactification, there is a natural partial compactification of the moduli space of convex $\rp^2$ structures on $S$ by adding regular cubic differentials over nodal Riemann surfaces at the boundary of the moduli space of Riemann surfaces.  The first author \cite{loftin02c} studied convex $\rp^2$ structures at the boundary of moduli from this point of view.  In particular, the residue of the cubic differential completely determines the structure of each end in terms of the holonomy and developing map.  (The structure of the ends developed in \cite{loftin02c} can be seen as a parabolic Higgs bundle developed over Riemann surfaces with punctures; this relies on Labourie's identification of the affine sphere equations and Hitchin's Higgs bundles discussed in the next subsection.) In addition, at least in generic situations, the map from pairs $(\Sigma,Q)$ to the holonomy and developing map data is continuous along a path toward the boundary of the moduli space.

From a purely geometric point of view, the ends of convex $\rp^2$ structures found in \cite{loftin02c} are discussed in Choi \cite{choi94b} and Marquis \cite{marquis12}.

Benoist-Hulin proved a partial converse to the main result of \cite{loftin02c} corresponding to the case of zero residue \cite{benoist-hulin13}.  They show that convex $\rp^2$ structures with finite volume with respect to the Hilbert metric have induced cubic differentials of residue zero (and so pole order of at most two). Marquis \cite{marquis12} had already shown that the finite-volume ends of convex $\rp^2$ structures correspond to the holonomy determined in \cite{loftin02c}. Benoist-Hulin's proof involves using Benz\'ecri's theorem on the cocompactness under the action of $\mathrm{PGL}(3,\re)$ of the space of pointed strictly convex projective domains \cite{benzecri60} to derive estimates for the affine differential geometric invariants associated to convex $\rp^2$ structures.

More recently, Benoist-Hulin have shown that on the unit disk $\mathcal D\subset\co$, cubic differentials bounded with respect to the hyperbolic metric correspond to $\rp^2$ structures which are Gromov-hyperbolic with respect to the Hilbert metric (or equivalently with respect to the Blaschke metric) \cite{benoist-hulin13a}.  One may think of the Gromov-hyperbolic convex $\rp^2$ structures as comprising an analog of the universal Teichm\"uller space. Thus this space can be parametrized by cubic differentials.

Also, Dumas-Wolf have shown that the space of convex $n$-gons in $\re^2$ modulo projective equivalence is homeomorphic to the space of polynomial cubic differentials of degree $n-3$ on $\co$ modulo holomorphic motions on $\co$ \cite{dumas-wolf13}.  For the triangle, the \c{T}i\c{t}eica-Calabi example corresponds to the cubic differential $dz^3$ on $\co$.

One may also attempt to identify boundary points of the moduli space of convex $\rp^2$ structures on $S$ by fixing a conformal structure $\Sigma$ and cubic differential $Q$.  Then we may consider a ray of cubic differentials $tQ$ as $t\to\infty$.  In \cite{loftin06}, the first author then studies the limiting holonomy configuration along all free loops in the sense of \cite{parreau00,inkang-kim05}. (This limiting holonomy is an analog of Thurston's boundary of Teichm\"uller space).

\subsubsection{Hitchin representations}\index{Hitchin representation}
If $\Sigma$ is marked Riemann surface with hyperbolic metric $\mu$ and with zero cubic differential, then the Blaschke metric is hyperbolic, and the associated hyperbolic affine sphere is an elliptic hyperboloid (by Theorem \ref{cubic-form-vanish} above).  This projects to an ellipse in $\rp^2$, which is projectively equivalent to a round disk.  Projective actions on the round disk are hyperbolic isometries under the Klein model of hyperbolic space, and thus the $\rp^2$ structure reduces to the point in Teichm\"uller space determined by $\Sigma$.  The holonomy representation then lies in $\mathrm{PSO}(2,1)\subset \mathrm{PSL}(3,\re)$.

If $S$ is a closed marked surface of genus at least 2, Theorem \ref{rp2-cubic} shows a convex $\rp^2$ structure on $S$ induces a conformal structure and a cubic differential $Q$. By varying $Q$ in a path to 0, we see that the holonomy representation $\sigma$ for the $\rp^2$ structure is homotopic to a representation in $\mathrm{PO}(2,1)$ corresponding to an element of Teichm\"uller space.  The connected component of the representation space
$$ {\rm Rep}(\pi_1S,\mathrm{PSL}(3,\re)) \equiv {\rm Hom}(\pi_1S,\mathrm{ PSL}(3,\re))/\mathrm{ PSL}(3,\re)$$
containing Teichm\"uller space in this way is called the \emph{Hitchin component} \cite{hitchin92}.  Choi-Goldman proved that all Hitchin representations are induced by convex $\rp^2$ structures \cite{choi-goldman93}, and so the Hitchin component coincides with the deformation space of convex $\rp^2$ structures on $S$.

For a fixed conformal structure $\Sigma$ on $S$, Hitchin uses Higgs bundles to parametrize this component by a pair of $(V,Q)$ of $V$ a holomorphic quadratic differential and $Q$ a holomorphic cubic differential on $\Sigma$.   Consider the Hitchin representation corresponding to $V=0$.  Labourie proved that Hitchin's cubic differential is, up to a constant factor, the same as the cubic differential of the convex $\rp^2$ structure \cite{labourie07}.  Thus it is natural to replace Hitchin's parametrization of this component by $H^0(\Sigma,K^2)\oplus H^0(\Sigma, K^3)$ with the total space of the vector bundle over Teichm\"uller space whose fiber over a Riemann surface $\Lambda$ is $H^0(\Lambda,K^3)$ the space of holomorphic cubic differentials over $\Lambda$.  See Section \ref{hod-sec} below.

The Higgs bundle formulation includes a natural twisted harmonic map from the surface to $\Sl3/\mathrm{ SO}(3,\re)$.  Labourie recognizes this map as a unimodular metric on $\re^3$ given by an orthogonal direct sum of the Blaschke metric on the hyperbolic affine sphere with a metric on the span of the position vector $f$ in which the $f$ is assigned length 1.

\subsubsection{Extensions}
Labourie also in \cite{labourie07} studies affine surfaces in $\re^3$ of constant Gaussian curvature 1 to parametrize representations of a surface group into the affine group $\mathrm{ ASL}(3,\re)$.

Recently Fox has developed a far-reaching generalization of both hyperbolic affine spheres and Einstein Weyl structures \cite{fox13}. The construction involves a pair of: 1) a projective equivalence class of torsion-free connections (such as the Blaschke connection $\na$), and 2) a conformal structure.  Fox defines the Einstein equation for such structures and studies their properties on surfaces.

\subsection{Parabolic and elliptic affine spheres}

\subsubsection{Motivation from mirror symmetry}\index{mirror symmetry}

An affine (or affine-flat) structure on a manifold $M$ is an $(X,G)$ structure with $X=\re^n$ and $G = \mathrm{AGL}(n,\re)$, the group of all affine automorphisms of $\re^n$\index{affine structure}.  Equivalently, an affine flat structure is given by a flat torsion-free connection on the tangent bundle.  The tangent bundle of an affine manifold $M$ naturally carries a complex structure: If $x=\{x^i\}$ are affine coordinates on $M$, and tangent vectors are represented by $y^i\frac{\partial}{\partial x^i}$, an affine coordinate change $x\mapsto Ax+b$ acts on the complex coordinates $z=\{z^i=x^i+\sqrt{-1}y^i\}$ by $z\mapsto Az+b$. We write $TM$ with this complex structure as $M^\co$.  A Riemannian metric $g_{ij} \,dx^idx^j$ on $M$ is called \emph{affine K\"ahler} or \emph{Hessian} if it is locally the Hessian of a convex function\index{Hessian metric}\index{affine K\"ahler metric}.  In this case the Hermitian metric $g_{ij}(x) \,dz^id\bar z^j$ is K\"ahler on $M^\co$.   $M$ is \emph{special affine} if the transition maps can be chosen in $G=\mathrm{ASL}(n,\re)$, the affine special linear group\index{special affine structure}\index{affine structure!special}. In this case there is a canonical volume form $\nu$ on $M$ preserved by the action of $G$.  On a special affine manifold $M$, an affine K\"ahler metric $g_{ij}\,dx^idx^j$ has volume form $\nu$ if and only if the corresponding K\"ahler metric on $M^\co$ is Calabi-Yau.  In local affine coordinates, for an affine potential function $\phi$, this condition is $$ \det \frac{\partial^2\phi}{\partial x^i \partial x^j} = 1.$$  The fibers of the tangent bundle are naturally special Lagrangian with the real part of the form $(\sqrt{-1})^{-n}dz^1\wedge \cdots \wedge dz^n$ as the calibration. We call such a metric on an affine manifold \emph{semi-flat Calabi-Yau}\index{semi-flat Calabi-Yau manifold}.

This construction plays an important part of the Strominger-Yau-Zaslow picture in mirror symmetry \cite{syz}.  Near the large complex structure limit in moduli, Calabi-Yau manifolds are expected to have the structure of a fibration with singularities.  The base of such a fibration is (outside a singular set) an affine manifold equipped with a semi-flat Calabi-Yau metric as above.  The fibers are special Lagrangian tori, which are quotients of the tangent spaces as above.  Then the mirror manifold is another Calabi-Yau manifold formed by considering the Legendre transform of the affine K\"ahler potential on the base and performing a Fourier transform on the fibers.  On the base, the Legendre transform creates a dual affine structure and affine K\"ahler potential, but preserves the metric.  In general there will also be instanton corrections, but we do not address this part of the theory.  See e.g. \cite{leung05} for details.

To construct the torus fibration as a quotient of the tangent bundle, a lattice bundle within the tangent bundle is needed to form each torus as a quotient.  The affine gluing maps then must preserve the lattice, and so must be in $\mathrm{ ASL}(n,\mathbb Z)$. In other words, the linear part of each affine gluing map must be represented by an integer matrix.

Gross-Wilson have made significant progress in dimension 2.  They have calculated the rescaled metric limits of certain elliptic $K3$ surfaces and found them to be semi-flat Calabi-Yau metrics on $S^2$ with singularities at 24 points \cite{gross-wilson00}.

Locally, a semi-flat Calabi-Yau metric is just a convex solution to $\det \phi_{ij} =1$. We have seen the graph of $\phi$ is a parabolic affine sphere, and in dimension 2, we can use Theorem \ref{aff-sph-integrate} above to construct parabolic affine spheres on a Riemann surface with cubic differential, as long as we can solve (\ref{aff-sph-eq-local}).  On a Riemann surface $\Sigma$ with cubic differential $Q$ and background metric $\mu$, we must solve
\begin{equation} \label{global-pas-eq}
 \Delta u + 2\|Q\|^2 e^{-2u}  -2\kappa = 0,
\end{equation}
to make $e^u\mu$ the Blaschke metric of a parabolic affine sphere.  On a compact Riemann surface, integrating the equation and Gauss-Bonnet rule out genus at least two and genus 1 except for the trivial case of $Q=0$. For genus 0, there are no nonzero cubic differentials.  If we allow $Q$ to have poles, we can find solutions on $\CP^1$ with an appropriate ansatz for the metric near the poles.  In \cite{loftin04}, for any cubic differential with poles of order 1, (\ref{global-pas-eq}) is solved for an appropriate background metric $\mu$.  The resulting parabolic affine sphere structure gives both an affine structure (the Blaschke connection $\na$ is affine-flat), and the potential for a semi-flat Calabi-Yau metric is given by the component of the parabolic affine sphere in the coordinate direction of the affine normal $\xi$.

\subsubsection{Affine structures}

The solution to (\ref{global-pas-eq}) follows from constructing an appropriate ansatz as a background metric and also sub and super solutions.  The analysis allows a detailed description of the geometry near each puncture.  In particular, the induced affine holonomy on a free loop around each pole of $Q$ is given by
$$x\mapsto \left(\begin{array}{cc} 1&1\\0&1 \end{array}\right)x,$$
which is the same local holonomy type as expected by Gross-Wilson's example.  It is much more difficult to determine when these solutions produce a flat Blaschke connection $\na$ with integral holonomy, as computing the holonomy requires solving the PDE (\ref{global-pas-eq}) and then integrating along paths away from the punctures to compute the parallel transport of $\na$.

The development of the parabolic affine sphere into $\re^3$ encodes not just the affine structure via the developing map, but also the full semi-flat Calabi-Yau data, which leads to the mirror manifold as well.  In particular, we may choose a transverse plane $X$ to the affine normal $\xi$ in $\re^3$. For the splitting of $\re^3$ induced by $X$ and $\xi$, choose projections $\pi_X$ and $\pi_\xi$ onto $X$ and $\langle\xi\rangle$ respectively. For the embedding $f$ of the parabolic affine sphere, $\pi_Xf$ is the affine developing map, while if $\pi_\xi f = \phi\xi$, $\phi$ is the affine K\"ahler potential for the metric.

For $(x^1,x^2)$ coordinates on $X$, the affine coordinates on the mirror manifold are $(y_1,y_2)=(\partial \phi/\partial x^1, \partial \phi / \partial x^2)$, and the mirror Hessian potential is  the Legendre transform $\phi^* = x^1y_1 + x^2y_2 -\phi$.

On a manifold $M$, a torsion-free flat connection $\na$ on $TM$ is equivalent to an affine structure via a construction of Koszul \cite{koszul65}.  The essential part is to reproduce the developing map from $\na$. Let $x\in M$ be a basepoint, and let $\tilde x\in\tilde M$ be a lift of $x$. If $v\in T_{\tilde x} \tilde M$, let $P_v$ be the vector field on $\tilde M$ constructed by parallel transport under $\na$, and let $\omega(P_v)=v$ define a $T_{\tilde x}\tilde M$-valued 1-form on $\tilde M$.  We check that $\omega$ is $\na$-parallel, and since $\na$ is torsion-free, $d\omega=0$. Then the map $$F(y) = \int_{\tilde x}^y \omega$$ defines the affine development $F\!:\tilde M \to T_{\tilde x}\tilde M$.  As above in Subsection \ref{hol-dev-subsection}, we may identify the developing map $F$ with the affine development $\pi_X f$ described above.

\subsubsection{Elliptic affine spheres} In \cite{lyz05,lyz-erratum}, the first author, Yau, and Zaslow study elliptic affine spheres structures over $\CP^1$ with cubic differentials $Q$ with exactly 3 poles, each of order 2.  In this case, the integrability condition is
\begin{equation} \label{global-eas-eq}
 \Delta u + 2\|Q\|^2 e^{-2u} + 2 e^u -2\kappa = 0.
\end{equation}
Note that the sign on $2e^u$ is not well suited for the maximum principle, and so we are able to solve the equation only for small nonzero $Q$.  This is a bit counter-intuitive, since the good term $\|Q\|^2e^{-2u}$ is required to be small.  But we may scale the equation to see that  $u$ solves (\ref{global-eas-eq}) if and only if $v=u-\log \delta$ solves
$$ \Delta v + 2\|Q\|^2e^{-2v} + 2\delta e^v - 2\kappa = 0.$$
This trick allows us to make the bad term $2\delta e^v$ small, and so to solve the equation. Existence is proved only for small nonzero $Q$.  It is possible that the space of  solutions to (\ref{global-eas-eq}) has a similar structure to that of (\ref{glob-ch2}) below. This would require extending the analytic techniques of \cite{huang-loftin-lucia12} to the setting of noncompact Riemann surfaces.

The elliptic affine sphere structure given by the data $(\Sigma,Q,e^u\mu)$, where $u$ solves (\ref{global-eas-eq}) naturally induces a real projective structure on $\Sigma$.  The projection $P(f)$ acts as the developing map, as in \ref{hol-dev-subsection} above.

\subsection{Minimal Lagrangian surfaces}

\subsubsection{Surfaces in $\CP^2$}

 To understand minimal Lagrangian immersions $f:\Sigma\to\CP^2$ of a compact Riemann surface, we should look at the usual
trichotomy of Riemann surfaces: genus $0$, genus $1$ or genus at least $2$. In the first case $Q_3=0$ and the harmonic
sequence theory of Eells \& Wood \cite{eells-wood83} (cf. \cite{chern-wolfson83}) applies.

For surfaces of genus $1$ the cubic differential $Q_3$ is necessarily a non-zero constant, in the sense that $dz^3$ is
globally well-defined and we may normalize $Q_3$. In principle all such tori
can be constructed using algebro-geometric data, the so-called \emph{spectral data}, by exploiting the loop group and integrable
system methods: see \cite{mcintosh03}.
This means all the solutions of the PDE \eqref{eq:LagToda} for this case can be written down in terms of abelian
functions determined by the spectral data.

For higher genus surfaces there has been no success through a direct analysis of the equation \eqref{eq:LagToda}. The only
approach which has yielded existence results is the gluing approach of Haskins \& Kapouleas \cite{haskins-kapouleas07}, motivated by the
desire to construct special Lagrangian cones in $\C^3$ whose links (intersection with $S^5$) are surfaces of odd genus at least
$3$. These links are minimal Legendrian surfaces in $S^5$: their projection down to $\CP^2$ (via the fibration
$S^5\to\CP^2$) is a minimal Lagrangian surface.
The construction glues together pieces of spheres with pieces of tori to get a smooth Legendrian, but not necessarily
minimal, surface of higher genus and then perturbs this into a genuine minimal Legendrian surface.

\subsubsection{Surfaces in $\CH^2$} \label{ch2-subsec}
The situation in $\CH^2$ is a bit different from $\CP^2$, in that we produce from an appropriate pair $(\Sigma,Q)$ of compact Riemann surface $\Sigma$ and cubic differential $Q$ an immersion of $\Sigma$ not into $\CH^2$ but into a (not necessarily complete) complex hyperbolic 4-manifold. Equivalently, we produce an immersion of the universal cover $\tilde\Sigma$ into $\CH^2$ which is equivariant under an induced discrete holonomy representation $\pi_1\Sigma\to \mathrm{SU}(2,1)$.  We would like to recover the geometry of this representation from the holomorphic data and the induced metric.  This will not always be possible, but we are able to produce many examples and to show many of these induced representations are convex-cocompact.

Given a Riemann surface $\Sigma$, a cubic differential $Q$, and a conformal metric $\mu$ on $\Sigma$, we need to solve (\ref{eq:GlobalLagToda}) to produce the immersion of $\tilde\Sigma$ into $\CH^2$ as a minimal Lagrangian surface.
\begin{equation} \label{glob-ch2}
\Delta_\mu u -2e^u -2\|Q\|^2_\mu e^{-2u} - 2\kappa_\mu=0.
\end{equation}
For compact $\Sigma$, (\ref{glob-ch2}) has no solutions on $\Sigma$ of genus 0 or 1, as integration and the Gauss-Bonnet Theorem give a contradiction.  Thus we restrict ourselves to the case of genus at least 2, and we assume the background metric $\mu$ is hyperbolic.

The following existence theorem was first proved by the authors in \cite{loftin-mcintosh13} and extended by the first author, Huang and Lucia in \cite{huang-loftin-lucia12}.  The combined existence result is
\begin{thm} \label{glob-ch2-exist}
Let $\Sigma$ be a compact Riemann surface with hyperbolic metric $\mu$ and nonzero holomorphic cubic differential $Q_0$.  For $t\in[0,\infty)$ consider $Q=tQ_0$. Then
\begin{itemize}
\item The unique solution for $t=0$ is $u=0$.
\item There is a $T_0\ge(\sqrt{27/4} \sup_\Sigma\|Q_0\|_\mu)^{-1}$ so that there is a continuous family $u_t$
 of smooth solutions to (\ref{glob-ch2}) for $t\in[0,T_0]$.
\item There is a finite $T>T_0$ so that there are no solutions to (\ref{glob-ch2}) for $t\ge T$.
\item For each $t\in(0,T_0)$, there is at least one other solution $v_t\neq u_t$ to (\ref{glob-ch2}).  These $v_t$ satisfy $\lim_{t\to0^+}\sup |v_t|=\infty$. It is not known whether these solutions $v_t$ form a continuous family.
\end{itemize}
\end{thm}

\begin{proof}[Idea of proof]
The proof of the lower bound on $T_0$, the upper bound $T$, and the uniqueness for $t=0$ are in \cite{loftin-mcintosh13}, where existence is proved by the method of sub and super solutions.  The existence of the continuous family $u_t$ on all of $[0,T_0]$ is proved in \cite{huang-loftin-lucia12}, following a continuity-method argument of Uhlenbeck \cite{uhlenbeck83}. The existence of $v_t$ is proved using a mountain pass argument in \cite{huang-loftin-lucia12}.
\end{proof}

Each solution to (\ref{glob-ch2}) produces a representation of $\pi_1\Sigma\to \mathrm{ SU}(2,1)$, and the geometry of such representations is closely related to the immersed minimal Lagrangian surface in $\CH^2$.  First of all, for $Q=0$, the induced minimal Lagrangian surface is simply a totally geodesic copy of $\RH^2\subset\CH^2$. These representations are called $\re$-Fuchsian (there are also $\co$-Fuchsian representations related to totally geodesic complex submanifolds  $\CH^1\subset\CH^2$).  We are interested in representations $\rho\!:\pi_1\Sigma\to \mathrm{ SU}(2,1)$ which are \emph{almost $\re$-Fuchsian} in the sense that there is an equivariant minimal Lagrangian surface whose normal bundle is diffeomorphic to $\CH^2$ under the exponential map.  (There is a similar notion for surfaces in $\RH^3$ \cite{uhlenbeck83,huang-wang13}.)  We show in  \cite{loftin-mcintosh13} that almost $\re$-Fuchsian representations are quasi-Fuchsian, or convex-cocompact.

\begin{prop}\cite{loftin-mcintosh13,loftin-mcintosh-inprep} On a compact Riemann surface with hyperbolic metric $\mu$, cubic differential $Q$, and  a solution $u$ to (\ref{glob-ch2}), the induced representation into $\mathrm{ SU}(2,1)$ is almost $\re$-Fuchsian if and only if the norm of the cubic differential $\|Q\|_{e^u\mu}$ is bounded above by $\sqrt{2}$ on all of $\Sigma$.
\end{prop}
\begin{proof}[Idea of proof]
That the exponential map of the normal bundle is an immersion is proved by a direct calculation in \cite{loftin-mcintosh13}, while in \cite{loftin-mcintosh-inprep} we show the induced complex-hyperbolic metric on the normal bundle is complete.
\end{proof}

In \cite{loftin-mcintosh13}, we show that if $t<(\sqrt{27/4} \sup_\Sigma\|Q_0\|_\mu)^{-1}$, then the solution $u_t$ satisfies $\|Q\|_{e^u\mu}\le 1/\sqrt{2}<\sqrt{2}$. By the proposition above, this shows that the induced representation is almost $\re$-Fuchsian.  It is not clear whether the almost $\re$-Fuchsian property extends all the way to solutions at $t=T_0$ (or beyond to solutions $v_t$ for $t$ near $T_0$). This is in contrast to the case of minimal surfaces in $\RH^3$, in which almost Fuchsian representations only occur for the continuous family $\tilde u_t$ for $t\in[0,T']$ with $T'<T_0$ \cite{uhlenbeck83}.  The geometry of the representations which are far from being $\re$-Fuchsian is unknown.  It is quite likely for example that the minimal Lagrangian surfaces associated to $v_t$ for $t$ near 0 may not be embedded in $\CH^2$ and the associated representations may not be discrete.

\subsection{The Toda equations and surface group representations}
Let $(\Sigma,\mu)$ be a closed hyperbolic surface of constant Gaussian curvature $-1$. Each of the four elliptic
versions of the $\fa_2^{(2)}$ Toda equations discussed above is the zero curvature equation for a flat $\spl(3,\C)$
connection on the trivial bundle over an open domain in $\C$. Each connection $1$-form is obtained by fixing the two sign
choices in
\begin{equation}\label{eq:Areal}
A = \begin{pmatrix}
\psi_z  & 0 & -\lambda e^\psi\\ Qe^{-2\psi} &  -\psi_z & 0\\ 0&  -\lambda e^\psi&0
\end{pmatrix}dz
+ \begin{pmatrix}
-\psi_{\bz}  & \pm{\bar Q}e^{-2\psi} & 0\\ 0 & \psi_\bz  & e^\psi\\ e^\psi & 0& 0
\end{pmatrix}d\bz,\quad\lambda = \pm 1.
\end{equation}
Each of the equations has a global version over $\Sigma$: one for each of the sign possibilties in
\begin{equation}\label{eq:realToda}
\Delta_\mu u+\lambda 2e^u \pm 2\|Q_3\|_\mu^2e^{-2u}  +2=0.
\end{equation}
Here $u:\Sigma\to \R$ is a globally smooth function and $Q_3$ is a cubic holomorphic differential over $\Sigma$.
When \eqref{eq:realToda} holds we have, in each local complex coordinate chart $(U_j,z_j)$ on $\Sigma$, a connection form $A_j$
obtained by taking $z=z_j$ in the appropriate local connection form above, and using the local expressions
\[
e^u\mu = 2e^{2\psi_j}|dz_j|^2,\quad Q_3 = Q_jdz_j^3,
\]
to replace $\psi$ and $Q$ in $A$. It is straightforward to check that in an overlapping chart $(U_k,z_k)$ the connexion
forms are related by the gauge transformation
\begin{equation}\label{eq:Agauge}
A_k =  c_{jk}^{-1} A_j c_{jk} + c_{jk}^{-1}dc_{jk}
\end{equation}
for the diagonal matrix
\[
c_{jk} = \diag\left( \frac{dz_j/dz_k}{|dz_j/dz_k|},\frac{\overline{dz_j/dz_k}}{|dz_j/dz_k|},1\right).
\]
Clearly this gives us a $1$-cocycle $\{(c_{jk},U_j,U_k)\}$ over $\Sigma$ with values in an $S^1$ subgroup $T\subset
\mathrm{ SL}(3,\C)$. Consequently we can apply the following well-known theorem.
\begin{thm}
Let $G$ be a Lie group with Lie algebra $\fg$, and $K\subset G$ a Lie subgroup. For a closed hyperbolic surface $\Sigma$
with universal cover $\caD$, the following items are equivalent:
\begin{enumerate}
\item an immersion $f:\caD\to G/K$ which is $\rho$-equivariant for some representation $\rho:\pi_1\Sigma\to G$;
\item a flat principal $G$-bundle $(P,\omega)$ over $\Sigma$ with a $K$-subbundle $Q\subset P$;
\item for an open cover $\{U_j\}$ of $\Sigma$, a $1$-cocycle $\{(c_{jk},U_j,U_k)\}$ with values in $K$ and
flat $\fg$-valued connection $1$-forms $(U_j,A_j)$ for which \eqref{eq:Agauge} holds.
\end{enumerate}
\end{thm}
The representation $\rho$ is only determined up to conjugacy: it is the holonomy representation for the flat connection
$\omega$ on $P$.

For the sign choices which yield the equations for minimal Lagrangian surfaces, we can take $G$ to be $\mathrm{ SU}(3)$ or $\mathrm{ SU}(2,1)$
as appropriate, and $K\supset T$ to be the isotropy subgroup of the base point $[e_3]$ used in the projective models for
$\CP^2$ and $\CH^2$ earlier. When the sign choice gives one of the two affine sphere cases, the flat connection is conjugate
(i.e., gauge equivalent by a constant gauge) to a $\spl(3,\R)$-connection, because
of \eqref{eq:sphereframe}.
Further, this conjugacy transforms the $1$-cocycle into an $\mathrm{ SL}(3,\R)$-valued cocycle (it just represents the Jacobi matrix
for real changes of coordinates over $\Sigma$), so that in the affine sphere cases the holonomy representation is conjugate
in $\mathrm{ SL}(3,\C)$ to an $\mathrm{ SL}(3,\R)$ connection. Now for $K$ we can take the stabilizer of the real line $[e_3]\in\RP^2$. After
conjugacy this contains the $1$-cocycle, and therefore we obtain a $\rho$ equivariant map to $\RP^2\simeq \mathrm{ SL}(3,\R)/K$. This
is the line generated by the affine normal $\xi$.

\subsection{Special Lagrangian submanifolds of Calabi-Yau Three-folds} \label{slag-subsec}

Four of the six types of surfaces we have studied can be viewed as dimension reductions of special Lagrangian submanifolds in Calabi-Yau three-folds, which are perhaps the prime objects of study in the Strominger-Yau-Zaslow picture of mirror symmetry.  One of the main unsolved issues in this theory is the structure of singularities of the affine base in dimension 3.  It is believed these singularities are of codimension 2, and so should represent a graph in a 3-manifold.  As we will see in the next paragraph, a slight extension of the parabolic affine sphere structure in dimension 2 provides a model along the edges of the graph.  The vertices, however, are much more complicated, and below we recount details of a program to describe the structure near trivalent vertices of the graph \cite{lyz05,lyz-erratum}.  There are still many open questions.

First of all, there is a trivial case.  A parabolic affine sphere in $\re^3$ is given by the graph of a convex solution $\phi(x^1,x^2)$ to a Monge-Amp\`ere equation $\phi_{11}\phi_{22}-\phi_{12}^2=1$. Then we may adjoin a third variable $x^3$ and consider $$\tilde\phi(x^1,x^2,x^3) = \phi + \sfrac12 (x^3)^2.$$  Then $\tilde\phi$ is a convex 3-dimensional solution to the Monge-Amp\`ere equation.  This construction allows the point singularities of the parabolic affine sphere structures above to be extended to singularities along line segments in dimension three.

Secondly, the proper affine spheres can be extended radially to produce a semi-flat Calabi-Yau metric conically on a three-fold.  The basic example is the quadric elliptic affine sphere $\mathcal S =\{\|x\|^2=1\}$ in $\re^3$. Then the function equal to $\frac12$ on $\mathcal S$ and radially homogeneous of order 2 is $\phi(x)=\frac12\|x\|^2$, which of course solves the Monge-Amp\`ere equation.  Baues-Cort\'es proved that this construction works for any elliptic affine sphere centered at the origin \cite{baues-cortes03}.   This construction promotes an elliptic affine sphere structure on $\CP^1$ minus 3 points to a semi-flat Calabi-Yau structure on a ball in $\re^3$ minus a trivalent vertex of a graph \cite{lyz05,lyz-erratum}.  The affine structure is not the same as that on $\re^3$, as there is some holonomy around nontrivial loops.  Recall that for the purposes of the Strominger-Yau-Zaslow picture, the holonomy should be integral.  But constructing this holonomy from the data on the elliptic affine sphere is quite difficult.  There are some partial results in \cite{dunajski-plansangkate09}, but no one knows how to reconstruct the full holonomy, or whether it is integral.

There is an analog of the result of Baues-Cort\'es for hyperbolic affine spheres in $\re^3$ as well.  Given a local hyperbolic affine sphere $H$ centered at the origin, let $r$ be the homogeneous function of degree 1 from the cone over $H$ to $(0,\infty)$ which is equal to 1 along $H$.  Then there is a function $f$ so that $\phi(x^1,x^2,x^3)=f(r)$ is convex and satisfies the Monge-Amp\`ere equation $\det \phi_{ij}=1$ near 0 in the cone over $H$ \cite{lyz-erratum}.  Thus hyperbolic affine sphere stuctures produced over $\CP^1$ minus three points in \cite{loftin02c} can be extended to a semi-flat Calabi-Yau structure on a ball minus a trivalent ``Y" vertex in dimension three.  The holonomy question remains open in this case; the holonomies of these examples seem to be unknown, and it is unclear whether there are nontrivial examples with integer holonomy among them.

There is also a cone construction which takes minimal Lagrangian surfaces in $\CP^2$ and to special Lagrangian cones in $\C^3$.  In this case, there is no reduction to an affine base, and the analog of the holonomy consideration above is not present: the complex structure on $\C^3$ is smooth.  The cone itself does have an isolated singularity at the cone point.  The minimal Lagrangian surfaces discussed here provide many examples of singularities of special Lagrangian 3-folds, but we are  far from classifying all possible singularities. For example, even in the case of cones over tori, there seems to be no good moduli space. It must incorporate a countable union of real families of every dimension: see for example \cite{carberry-mcintosh04}.

\section{Higher-order differentials} \label{hod-sec}

Let $S$ be an oriented closed marked surface of genus at least two.  For any split real simple Lie group $G$ with trivial center, Hitchin uses Higgs bundles to parametrize a component of the representation space ${\rm Rep}(\pi_1S,G)$ \cite{hitchin92}.  Then for a conformal structure $\Sigma$ on $S$, the representations in the Hitchin component are parametrized by holomorphic sections
$$ \bigoplus_{i=1}^r H^0(\Sigma, K^{m_i+1})$$
where $r$ is the rank of $G$ and the $m_i$ are the exponents of $G$.  Since one is always an exponent of $G$, we set $m_1=1$ and there is natural quadratic differential in the first slot.  Considering the quadratic differential as a cotangent vector to Teichm\"uller space $\mathcal T$, we can instead consider bundle $\mathcal B_G$ over Teichm\"uller space whose fibers are $$\bigoplus_{i=2}^r H^0(\Sigma, K^{m_i+1}),$$
and consider the Hitchin representation corresponding to $\Sigma\in\mathcal T$ and holomorphic sections of the form $$(0,\sigma_2,\dots,\sigma_r), \qquad \sigma_i\in H^0(\Sigma,K^{m_i+1}).$$
This forms a map from $\mathcal B_G$ to the Hitchin component.
For $G=\mathrm{ PSL}(n+1,\re)$, we have each exponent $m_i=i$ and Labourie proved this map is surjective \cite{labourie08}.  This surjectivity was extended to several other groups by Baraglia \cite{baraglia-thesis}.

In the case of $G=\mathrm{ PSL}(3,\re)$, we have $m_2=1$ and so $\mathcal B_G$ has fibers $H^0(\Sigma,K^3)$. Labourie has shown that the Hitchin's cubic differential is up to a constant the same as the cubic form from the hyperbolic affine sphere, and that the corresponding map
$$\mathcal H\!:\mathcal B_G\to {\rm Rep}(\pi_1\Sigma,G)$$
is a homeomorphism \cite{labourie07}.  The proof heavily involves the machinery of the affine differential geometry: A Hitchin representation determines a convex $\rp^2$ structure by \cite{choi-goldman93}. The convex $\rp^2$ structure induces a hyperbolic affine sphere, which comes with a conformal structure and a holomorphic cubic differential \cite{loftin01,labourie97,labourie07}.  This conformal structure and cubic differential (up to a constant) coincide with Hitchin's.

Labourie \cite{labourie06a} and Goldman-Wentworth \cite{goldman-wentworth07} conjecture that this map $\mathcal H$ is a homeomorphism for general $ G$.  This is an open question except in the case of $G=\mathrm{ PSL}(2,\re)$ and $G=\mathrm{ PSL}(3,\re)$.

In his thesis, Baraglia has investigated Hitchin representations for the other rank-2 simple Lie groups $\mathrm{ PSp}(4,\re)$ and $G_2$ \cite{baraglia-thesis}.  In this case, Hitchin's Higgs bundles can be shown to correspond to the Toda lattice theory much as we have explained above.  Baraglia also explains the case of $\mathrm{ PSL}(3,\re)$ from the point of view of the Toda theory, and thus offers a new perspective in addition to \cite{loftin01,labourie97,labourie07}.  In particular, Baraglia has developed a theory of quartic differentials for Hitchin representations into $\mathrm{ PSp}(4,\re)$, and a theory of sextic differentials for Hitchin representations into the split real form of $G_2$.  The Toda theory for compact real forms has been applied to minimal surfaces in \cite{bolton-pw95}.  The geometry behind Toda-type representations for other real forms of general simple Lie groups is not well-developed.

\frenchspacing
\bibliographystyle{abbrv}

\begin{thebibliography}{10}

\bibitem{baraglia-thesis}
D.~Baraglia.
\newblock {\em $G_2$ Geometry and Integrable Systems}.
\newblock PhD thesis, University of Oxford, 2009.
\newblock arXiv:1002.1767.

\bibitem{baues-cortes03}
O.~Baues and V.~Cort{\'e}s.
\newblock Proper affine hyperspheres which fiber over projective special
  {K}\"ahler manifolds.
\newblock {\em Asian J. Math.}, 7(1):115--132, 2003.

\bibitem{benoist-hulin13}
Y.~Benoist and D.~Hulin.
\newblock Cubic differentials and finite volume convex projective surfaces.
\newblock {\em Geom. Topol.}, 17(1):595--620, 2013.

\bibitem{benoist-hulin13a}
Y.~Benoist and D.~Hulin.
\newblock Cubic differentials and hyperbolic convex sets.
\newblock preprint, 2013.

\bibitem{benzecri60}
J.~P. Benz{\'e}cri.
\newblock Sur les vari{\'e}t{\'e}s localement affines et projectives.
\newblock {\em Bull{\'e}tin de la Soci{\'e}t{\'e} Math{\'e}matique de France},
  88:229--332, 1960.

\bibitem{bobenko94}
A.~I. Bobenko.
\newblock Surfaces in terms of {$2$} by {$2$} matrices. {O}ld and new
  integrable cases.
\newblock In {\em Harmonic maps and integrable systems}, Aspects Math., E23,
  pages 83--127. Friedr. Vieweg, Braunschweig, 1994.

\bibitem{bobenko-schief99}
A.~I. Bobenko and W.~K. Schief.
\newblock Affine spheres: discretization via duality relations.
\newblock {\em Experiment. Math.}, 8(3):261--280, 1999.

\bibitem{bolton-pw95}
J.~Bolton, F.~Pedit, and L.~Woodward.
\newblock Minimal surfaces and the affine {T}oda field model.
\newblock {\em J. Reine Angew. Math.}, 459:119--150, 1995.

\bibitem{burstall95}
F.~E. Burstall.
\newblock Harmonic tori in spheres and complex projective spaces.
\newblock {\em J. Reine Angew. Math.}, 469:149--177, 1995.

\bibitem{burstall-pedit94}
F.~E. Burstall and F.~Pedit.
\newblock Harmonic maps via {A}dler-{K}ostant-{S}ymes theory.
\newblock In {\em Harmonic maps and integrable systems}, Aspects Math., E23,
  pages 221--272. Friedr. Vieweg, Braunschweig, 1994.

\bibitem{burstall-pedit95}
F.~E. Burstall and F.~Pedit.
\newblock Dressing orbits of harmonic maps.
\newblock {\em Duke Math. J.}, 80(2):353--382, 1995.

\bibitem{calabi67}
E.~Calabi.
\newblock Minimal immersions of surfaces in {E}uclidean spheres.
\newblock {\em J. Differential Geometry}, 1:111--125, 1967.

\bibitem{calabi72}
E.~Calabi.
\newblock Complete affine hyperspheres {I}.
\newblock {\em Instituto Nazionale di Alta Matematica Symposia Mathematica},
  10:19--38, 1972.

\bibitem{calabi88}
E.~Calabi.
\newblock Convex affine maximal surfaces.
\newblock {\em Results Math.}, 13(3-4):199--223, 1988.
\newblock Reprinted in {\it Affine Differentialgeometrie} [(Oberwolfach, 1986),
  199--223, Tech. Univ. Berlin, Berlin, 1988].

\bibitem{calabi-nirenberg74}
E.~Calabi and L.~Nirenberg.
\newblock unpublished, 1974.

\bibitem{carberry-mcintosh04}
E.~Carberry and I.~McIntosh.
\newblock Minimal {L}agrangian 2-tori in {$\Bbb C\Bbb P^2$} come in real
  families of every dimension.
\newblock {\em J. London Math. Soc. (2)}, 69(2):531--544, 2004.

\bibitem{cheng-yau77}
S.-Y. Cheng and S.-T. Yau.
\newblock On the regularity of the {M}onge-{A}mp{\`e}re equation
  $\det((\partial^2u/\partial x^i\partial x^j))={F}(x,u)$.
\newblock {\em Communications on Pure and Applied Mathematics}, 30:41--68,
  1977.

\bibitem{cheng-yau86}
S.-Y. Cheng and S.-T. Yau.
\newblock Complete affine hyperspheres. part {I}. {T}he completeness of affine
  metrics.
\newblock {\em Communications on Pure and Applied Mathematics}, 39(6):839--866,
  1986.

\bibitem{chern-wolfson83}
S.~S. Chern and J.~G. Wolfson.
\newblock Minimal surfaces by moving frames.
\newblock {\em Amer. J. Math.}, 105(1):59--83, 1983.

\bibitem{choi94b}
S.~Choi.
\newblock Convex decompositions of real projective surfaces. {II}: Admissible
  decompositions.
\newblock {\em Journal of Differential Geometry}, 40(2):239--283, 1994.

\bibitem{choi-goldman93}
S.~Choi and W.~M. Goldman.
\newblock Convex real projective structures on closed surfaces are closed.
\newblock {\em Proceedings of the American Mathematical Society},
  118(2):657--661, 1993.

\bibitem{d-goldman}
M.-R. Darvishzadeh and W.~M. Goldman.
\newblock Deformation spaces of convex real projective and hyperbolic affine
  structures.
\newblock {\em Journal of the Korean Mathematical Society}, 33:625--638, 1996.

\bibitem{dorfmeister-pedit-wu98}
J.~Dorfmeister, F.~Pedit, and H.~Wu.
\newblock Weierstrass type representation of harmonic maps into symmetric
  spaces.
\newblock {\em Comm. Anal. Geom.}, 6(4):633--668, 1998.

\bibitem{drinfeld-sokolov84}
V.~G. Drinfel{\cprime}d and V.~V. Sokolov.
\newblock Lie algebras and equations of {K}orteweg-de {V}ries type.
\newblock In {\em Current problems in mathematics, {V}ol. 24}, Itogi Nauki i
  Tekhniki, pages 81--180. Akad. Nauk SSSR Vsesoyuz. Inst. Nauchn. i Tekhn.
  Inform., Moscow, 1984.

\bibitem{dumas-wolf13}
D.~Dumas and M.~Wolf.
\newblock private communication.

\bibitem{dunajski-plansangkate09}
M.~Dunajski and P.~Plansangkate.
\newblock Strominger-{Y}au-{Z}aslow geometry, affine spheres and {P}ainlev\'e
  {III}.
\newblock {\em Comm. Math. Phys.}, 290(3):997--1024, 2009.

\bibitem{eells-wood83}
J.~Eells and J.~C. Wood.
\newblock Harmonic maps from surfaces to complex projective spaces.
\newblock {\em Adv. in Math.}, 49(3):217--263, 1983.

\bibitem{ferrer-mm99}
L.~Ferrer, A.~Mart{\'\i}nez, and F.~Mil\'an.
\newblock An extension of a theorem by {K}. {J}{\"o}rgens and a maximum
  principle at infinity for parabolic affine spheres.
\newblock {\em Mathematische Zeitschrift}, 230(3):471--486, 1999.

\bibitem{fox13}
D.~J. Fox.
\newblock Einstein-like geometric structures on surfaces.
\newblock {\em Ann. Sc. Norm. Super. Pisa Cl. Sci. (5)}, 12:499--585, 2013.

\bibitem{gigena81}
S.~Gigena.
\newblock On a conjecture by {E}. {C}alabi.
\newblock {\em Geometriae Dedicata}, 11:387--396, 1981.

\bibitem{goldman90a}
W.~M. Goldman.
\newblock Convex real projective structures on compact surfaces.
\newblock {\em Journal of Differential Geometry}, 31:791--845, 1990.

\bibitem{goldman90b}
W.~M. Goldman.
\newblock The symplectic geometry of affine connections on surfaces.
\newblock {\em Journal f{\"u}r die reine und agnewandte Mathematik},
  407:126--159, 1990.

\bibitem{goldman-wentworth07}
W.~M. Goldman and R.~A. Wentworth.
\newblock Energy of twisted harmonic maps of {R}iemann surfaces.
\newblock In {\em In the tradition of {A}hlfors-{B}ers. {IV}}, volume 432 of
  {\em Contemp. Math.}, pages 45--61. Amer. Math. Soc., Providence, RI, 2007.

\bibitem{greene-svy90}
B.~Greene, A.~Shapere, C.~Vafa, and S.-T. Yau.
\newblock Stringy cosmic strings and noncompact {C}alabi-{Y}au manifolds.
\newblock {\em Nuclear Physics B}, 337(1):1--36, 1990.

\bibitem{gross-wilson00}
M.~Gross and P.~M.~H. Wilson.
\newblock Large complex structure limits of {K}3 surfaces.
\newblock {\em Journal of Differential Geometry}, 55(3):475--546, 2000.

\bibitem{harvey-lawson82}
R.~Harvey and H.~B. Lawson, Jr.
\newblock Calibrated geometries.
\newblock {\em Acta Math.}, 148:47--157, 1982.

\bibitem{haskins-kapouleas07}
M.~Haskins and N.~Kapouleas.
\newblock Special {L}agrangian cones with higher genus links.
\newblock {\em Invent. Math.}, 167(2):223--294, 2007.

\bibitem{helgason01}
S.~Helgason.
\newblock {\em Differential geometry, {L}ie groups, and symmetric spaces},
  volume~34 of {\em Graduate Studies in Mathematics}.
\newblock American Mathematical Society, Providence, RI, 2001.
\newblock Corrected reprint of the 1978 original.

\bibitem{hitchin92}
N.~J. Hitchin.
\newblock Lie groups and {T}eichm\"uller space.
\newblock {\em Topology}, 31(3):449--473, 1992.

\bibitem{huang-loftin-lucia12}
Z.~Huang, J.~Loftin, and M.~Lucia.
\newblock Holomorphic cubic differentials and minimal lagrangian surfaces in
  $ch^2$.
\newblock arXiv:1201.3941, to appear, Mathematical Research Letters.

\bibitem{huang-wang13}
Z.~Huang and B.~Wang.
\newblock On almost-{F}uchsian manifolds.
\newblock {\em Trans. Amer. Math. Soc.}, 365(9):4679--4698, 2013.

\bibitem{inkang-kim05}
I.~Kim.
\newblock Compactification of strictly convex real projective structures.
\newblock {\em Geom. Dedicata}, 113:185--195, 2005.

\bibitem{koszul65}
J.~L. Koszul.
\newblock Vari\'et\'es localement plates et convexit\'e.
\newblock {\em Osaka Journal of Mathematics}, 2:285--290, 1965.

\bibitem{kuiper54}
N.~Kuiper.
\newblock On convex locally-projective spaces.
\newblock In {\em Convegno Int. Geometria Diff.}, pages 200--213, Italy, 1954.

\bibitem{labourie97}
F.~Labourie.
\newblock in Proceedings of the GARC Conference in Differential Geometry, Seoul
  National University, Fall 1997, 1997.

\bibitem{labourie06a}
F.~Labourie.
\newblock Anosov flows, surface groups and curves in projective space.
\newblock {\em Invent. Math.}, 165(1):51--114, 2006.

\bibitem{labourie07}
F.~Labourie.
\newblock Flat projective structures on surfaces and cubic holomorphic
  differentials.
\newblock {\em Pure and Applied Mathematics Quarterly}, 3(4):1057--1099, 2007.
\newblock Special issue in honor of Grisha Margulis, Part 1 of 2.

\bibitem{labourie08}
F.~Labourie.
\newblock Cross ratios, {A}nosov representations and the energy functional on
  {T}eichm\"uller space.
\newblock {\em Ann. Sci. \'Ec. Norm. Sup\'er. (4)}, 41(3):437--469, 2008.

\bibitem{leung05}
N.~C. Leung.
\newblock Mirror symmetry without corrections.
\newblock {\em Comm. Anal. Geom.}, 13(2):287--331, 2005.

\bibitem{li89}
A.~M. Li.
\newblock Affine maximal surfaces and harmonic functions.
\newblock In {\em Differential geometry and topology (Tianjin, 1986--87)},
  volume 1369 of {\em Lecture Notes in Math.}, pages 142--151. Springer,
  Berlin, 1989.

\bibitem{li90}
A.-M. Li.
\newblock Calabi conjecture on hyperbolic affine hyperspheres.
\newblock {\em Mathematische Zeitschrift}, 203:483--491, 1990.

\bibitem{li92}
A.-M. Li.
\newblock Calabi conjecture on hyperbolic affine hyperspheres (2).
\newblock {\em Mathematische Annalen}, 293:485--493, 1992.

\bibitem{qli13}
Q.~Li.
\newblock Teichm\"uller space is totally geodesic in {G}oldman space.
\newblock arXiv:1301.1442, 2013.

\bibitem{loftin01}
J.~Loftin.
\newblock Affine spheres and convex $\mathbb{RP}^n$ manifolds.
\newblock {\em American Journal of Mathematics}, 123(2):255--274, 2001.

\bibitem{loftin06}
J.~Loftin.
\newblock Flat metrics, cubic differentials and limits of projective
  holonomies.
\newblock {\em Geom. Dedicata}, 128:97--106, 2007.

\bibitem{loftin10}
J.~Loftin.
\newblock Survey on affine spheres.
\newblock In {\em Handbook of geometric analysis, {N}o. 2}, volume~13 of {\em
  Adv. Lect. Math. (ALM)}, pages 161--191. Int. Press, Somerville, MA, 2010.

\bibitem{loftin-mcintosh-inprep}
J.~Loftin and I.~McIntosh.
\newblock in preparation.

\bibitem{loftin-mcintosh13}
J.~Loftin and I.~McIntosh.
\newblock Minimal {L}agrangian surfaces in {$\Bbb{CH}^2$} and representations
  of surface groups into {$SU(2,1)$}.
\newblock {\em Geom. Dedicata}, 162:67--93, 2013.

\bibitem{loftin-tsui08}
J.~Loftin and M.-P. Tsui.
\newblock Ancient solutions of the affine normal flow.
\newblock {\em J. Differential Geom.}, 78(1):113--162, 2008.

\bibitem{lyz05}
J.~Loftin, S.-T. Yau, and E.~Zaslow.
\newblock Affine manifolds, {SYZ} geometry and the ``{Y}'' vertex.
\newblock {\em J. Differential Geom.}, 71(1):129--158, 2005.
\newblock erratum, 2008, arXiv:math/0405061.

\bibitem{lyz-erratum}
J.~Loftin, S.-T. Yau, and E.~Zaslow.
\newblock Erratum to affine manifolds, {S}{Y}{Z} geometry and the ``{Y}"
  vertex.
\newblock available at http://andromeda.rutgers.edu/$\sim$loftin/, 2008.

\bibitem{loftin02c}
J.~C. Loftin.
\newblock The compactification of the moduli space of $\mathbb{RP}^2$ surfaces,
  {I}.
\newblock {\em Journal of Differential Geometry}, 68(2):223--276, 2004.
\newblock math.DG/0311052.

\bibitem{loftin04}
J.~C. Loftin.
\newblock Singular semi-flat {C}alabi-{Y}au metrics on ${S}^2$.
\newblock {\em Communications in Analysis and Geometry}, 13(2):333--361, 2005.
\newblock math.DG/0403218.

\bibitem{marquis12}
L.~Marquis.
\newblock Surface projective convexe de volume fini.
\newblock {\em Ann. Inst. Fourier (Grenoble)}, 62(1):325--392, 2012.

\bibitem{mcintosh98}
I.~McIntosh.
\newblock On the existence of superconformal {$2$}-tori and doubly periodic
  affine {T}oda fields.
\newblock {\em J. Geom. Phys.}, 24(3):223--243, 1998.

\bibitem{mcintosh03}
I.~McIntosh.
\newblock Special {L}agrangian cones in {$\Bbb C\sp 3$} and primitive harmonic
  maps.
\newblock {\em J. London Math. Soc. (2)}, 67(3):769--789, 2003.

\bibitem{nomizu-sasaki}
K.~Nomizu and T.~Sasaki.
\newblock {\em Affine Differential Geometry: Geometry of Affine Immersions}.
\newblock Cambridge University Press, 1994.

\bibitem{parreau00}
A.~Parreau.
\newblock {\em D{\'e}g{\'e}n{\'e}rescence de sous-groupes discrets des groupes
  de {L}ie semisimples et actions de groupes sur des immeubles affines}.
\newblock PhD thesis, Universit{\'e} de Paris-sud, 2000.

\bibitem{sasaki80}
T.~Sasaki.
\newblock Hyperbolic affine hyperspheres.
\newblock {\em Nagoya Mathematical Journal}, 77:107--123, 1980.

\bibitem{schoen-yau94}
R.~Schoen and S.-T. Yau.
\newblock {\em Lectures on Differential Geometry}.
\newblock International Press, 1994.

\bibitem{segal89}
G.~Segal.
\newblock Loop groups and harmonic maps.
\newblock In {\em Advances in homotopy theory ({C}ortona, 1988)}, volume 139 of
  {\em London Math. Soc. Lecture Note Ser.}, pages 153--164. Cambridge Univ.
  Press, Cambridge, 1989.

\bibitem{simon-wang93}
U.~Simon and C.-P. Wang.
\newblock Local theory of affine 2-spheres.
\newblock In {\em Differential Geometry: Riemannian geometry (Los Angeles, CA,
  1990)}, volume 54-3 of {\em Proceedings of Symposia in Pure Mathematics},
  pages 585--598. American Mathematical Society, 1993.

\bibitem{syz}
A.~Strominger, S.-T. Yau, and E.~Zaslow.
\newblock Mirror symmetry is {T}-duality.
\newblock {\em Nuclear Physics}, B479:243--259, 1996.
\newblock hep-th/9606040.

\bibitem{terng83}
C.-L. Terng.
\newblock Affine minimal surfaces.
\newblock In {\em Seminar on minimal submanifolds}, volume 103 of {\em Ann. of
  Math. Stud.}, pages 207--216. Princeton Univ. Press, Princeton, NJ, 1983.

\bibitem{tzitzeica07}
G.~Tzitz{\'e}ica.
\newblock Sur une nouvelle classe de surfaces.
\newblock {\em Comptes redus hebdomadaires des s{\'e}ances de l'{A}cad{\'e}mie
  des sciences}, 144:1257--1259, 1907.

\bibitem{tzitzeica08}
G.~Tzitz{\'e}ica.
\newblock Sur une nouvelle classe de surfaces.
\newblock {\em Rend. Circ. mat. Palermo}, 25:180--187, 1908.

\bibitem{tzitzeica09}
G.~Tzitz{\'e}ica.
\newblock Sur une nouvelle classe de surfaces, 2$^{\rm \grave{e}me}$ partie.
\newblock {\em Rend. Circ. mat. Palermo}, 25:210--216, 1909.

\bibitem{uhlenbeck89}
K.~Uhlenbeck.
\newblock Harmonic maps into {L}ie groups: classical solutions of the chiral
  model.
\newblock {\em J. Differential Geom.}, 30(1):1--50, 1989.

\bibitem{uhlenbeck83}
K.~K. Uhlenbeck.
\newblock Closed minimal surfaces in hyperbolic {$3$}-manifolds.
\newblock In {\em Seminar on minimal submanifolds}, volume 103 of {\em Ann. of
  Math. Stud.}, pages 147--168. Princeton Univ. Press, Princeton, NJ, 1983.

\bibitem{wang91}
C.-P. Wang.
\newblock Some examples of complete hyperbolic affine $2$-spheres in
  $\mathbb{R}^3$.
\newblock In {\em Global Differential Geometry and Global Analysis}, volume
  1481 of {\em Lecture Notes in Mathematics}, pages 272--280. Springer-Verlag,
  1991.

\bibitem{wood84}
J.~C. Wood.
\newblock Holomorphic differentials and classification theorems for harmonic
  maps and minimal immersions.
\newblock In {\em Global {R}iemannian geometry ({D}urham, 1983)}, Ellis Horwood
  Ser. Math. Appl., pages 168--175. Horwood, Chichester, 1984.

\end{thebibliography}
\def\cprime{$'$}

\end{document}